\documentclass[11pt,reqno]{amsart}
\usepackage{amsmath,amsfonts,amssymb,amscd,amsthm,amsbsy,bbm, epsf,calc,  pdfsync}
\usepackage{color}
\usepackage{datetime}

\usepackage{thmtools, thm-restate}
\usepackage{thm-patch, thm-kv, thm-autoref}
\usepackage[colorlinks=true, urlcolor=blue, linkcolor=blue, citecolor=black]{hyperref}

\usepackage{geometry}
\numberwithin{equation}{section}
\newcounter{hours}\newcounter{minutes}

\theoremstyle{plain}
\declaretheorem[title=Theorem, parent=section]{thm}
\declaretheorem[title=Lemma,sibling=thm]{lem}
\declaretheorem[title=Corollary,sibling=thm]{cor}
\declaretheorem[title=Proposition,sibling=thm]{prop}
\declaretheorem[title=Definition,sibling=thm]{DEF}
\declaretheorem[title=Remark,sibling=thm]{rem}
\declaretheorem[title=Assumption,sibling=thm]{assume}

\def\I{\mathcal I}
\def\M{{\mathcal M}}
\def\real{{\mathbb R}}

\def\integer{{\mathbb Z}}
\def\Indicator{{\mathbbm{1}}}
\def\ep{\varepsilon}
\def\al{\alpha}
\def\del{\delta}

\def\Om{\Omega}
\def\gam{\gamma} 
\def\lam{\lambda}
\def\Lam{\Lambda}
\def\sig{\sigma}
\def\grad{\nabla}

\def\Tr{\textnormal{Tr}}
\def\Id{\textnormal{Id}}
\def\union{\cup}

\def\intersect{\cap}

\DeclareMathOperator*{\osc}{osc}

\newcommand{\abs}[1]{\left| #1 \right|}
\newcommand{\norm}[1]{\lVert#1\rVert}

\newcommand{\Div}{\textnormal{div}}

\begin{document}

\title{Neumann Homogenization via Integro-Differential Operators, Part 2: Singular Gradient Dependence}

\begin{abstract}
We continue the program initiated in a previous work, applying integro-differential methods to Neumann Homogenization problems.  We target the case of linear periodic equations with a singular drift, which includes (with some regularity assumptions) divergence equations with \emph{non-co-normal} oscillatory Neumann conditions.  Our analysis focuses on an induced integro-differential homogenization problem on the boundary of the domain.  Also, we use homogenization results for regular Dirichlet problems to build barriers for the oscillatory Neumann problem with the singular gradient term.  We note that our method allows to recast some existing results for fully nonlinear Neumann homogenization into this same framework. 
\end{abstract}

\author{Nestor Guillen}
\author{Russell W. Schwab}

\address{Department of Mathematics\\
University of Massachusetts, Amherst\\
Amherst, MA  90095}
\email{nguillen@math.umass.edu}

\address{Department of Mathematics\\
Michigan State University\\
619 Red Cedar Road \\
East Lansing, MI 48824}
\email{rschwab@math.msu.edu}

\date{Monday 20th November 2017, revised version}

\thanks{
The work of N. Guillen was partially supported by NSF DMS-1201413. We would like to thank Zhongwei Shen for suggesting we try our method from \cite{GuSc-2014NeumannHomogPart1DCDS-A} for the oblique derivative, divergence equation.  We would like to thank both of the anonymous referees for helpful and thoughtful comments which hopefully led to an improved presentation of this result, especially the addition of Section \ref{sec:DivergenceFormEq} and references for boundary layers in Section \ref{sec:BackgrounAndAssumptions}.
}


\maketitle
\baselineskip=14pt
\pagestyle{plain}              
\pagestyle{headings}		
\markboth{Nestor Guillen, Russell Schwab}{Neumann Homogenization via Integro-Differential Operators, Part 2}

\section{Introduction}\label{sec:Intro}

\subsection{Statement of The Main Result}
In this paper, we study the periodic homogenization of linear elliptic equations with a singular drift with oscillatory Neumann conditions.  In particular, we study the $\ep\to0$ behavior of the solutions, $u^\ep$, to the equation
\begin{align}\label{eqIntro:MainEpScale}
	\begin{cases}
		\Tr(A(\frac{x}{\ep})D^2u^\ep(x)) + \frac{1}{\ep}B(\frac{x}{\ep})\cdot \grad u^\ep = 0\ &\text{in}\ \Sigma^1\\
		u^\ep = 0\ &\text{on}\ \Sigma_1\\
		\partial_n u^\ep(x)=g(\frac{x}{\ep})\ &\text{on}\ \Sigma_0.
	\end{cases}
\end{align} 
The equation is posed in an infinite strip domain, with a normal vector given by $n$, which is
\begin{align*}
	\Sigma^1 \subset \real^{d+1},\  \Sigma^1 = \{x\in\real^{d+1}\ :\ 0<x\cdot n<1 \}.
\end{align*}
The ``top'' and ``bottom'' boundaries are respectively
\begin{align*}
	&\Sigma_1 \subset \real^{d+1},\  \Sigma_1 = \{x\in\real^{d+1}\ :\ x\cdot n=1\},\\
	&\Sigma_0 \subset \real^{d+1},\  \Sigma_0 = \{x\in\real^{d+1}\ :\ x\cdot n=0\}.
\end{align*}
The main goal is to prove that the ``nonlinear averaging'' effects of the equation in the interior and on the boundary are ``compatible'', and that $u^\ep$ will converge uniformly to $\bar u$, that is the unique solution of
\begin{align}\label{eqIntro:MainEffective}
	\begin{cases}
		\Tr(\bar AD^2\bar u(x)) = 0\ &\text{in}\ \Sigma^1\\
		\bar u = 0\ &\text{on}\ \Sigma_1\\
		\partial_n \bar u(x)=\bar g\ &\text{on}\ \Sigma_0.
	\end{cases}
\end{align} 
The matrix $\bar A$ is a unique constant, and it is the same one obtained in the homogenization of (\ref{eqIntro:MainEpScale}) without oscillatory boundary data (see Section \ref{sec:Perturbed} and \ref{eqPerturb:ABarDef}).  Due to the flat geometry of $\Sigma^1$, $\bar g$ will also be a unique constant (but in more general situations would be a function of the normal vector, $n(x)$).  Thus, this is simply another way of saying that there is a unique affine function, $\bar u$, such that $u^\ep$ converges to $\bar u$ uniformly in $\overline{\Sigma^1}$.

We would like to alert the reader to a standard assumption involving homogenization of equations like (\ref{eqIntro:MainEpScale}) as well as to the choice to study the non-physical boundary condition for the normal derivative of $u^\ep$, instead of the physical boundary condition typically associated to the co-normal derivative of $u^\ep$.  In some situations, second order equations arise from conserved integral quantities, which take on a divergence structure (and hence a co-normal boundary condition is often considered the natural and physical one).  Equally naturally, in other situations, the second order equations arise from probabilistic considerations, via expected terminal values of diffusion processes and Ito's formula.  Our motivation for equation (\ref{eqIntro:MainEpScale}) is in the second category, that of studying diffusion on large time-space scales with a drift that has the same strength as the diffusion.  The boundary condition means that we are looking at a diffusion that reflects off of the boundary, $\Sigma_0$, and (\ref{eqIntro:MainEpScale}) records a running cost seen by the diffusion, measured in its local time, given by the data $g(x/\ep)$ (see e.g. \cite[Sec. 8]{SatoUeno-1965MultiDimDiffBoundryMarkov} and/or \cite[Chp. IV, Sec. 7]{IkedaWatanabe-1981SDE}).  This reflection condition of the underlying diffusion means that the probabilistic boundary condition is the normal derivative, and \emph{not} the co-normal derivative.  Thus our study arises from a different source than the more commonly studied divergence structure equations involving the co-normal derivative.  Of course, since there are multiple coupled oscillatory space scales in (\ref{eqIntro:MainEpScale}), one seeks the question of whether or not solutions will converge to those of a simpler equation.  In probabilistic terms, when the drift is as in (\ref{eqIntro:MainEpScale}), it is possible that trajectories will go ``ballistic'' and exit the domain quickly due to the drift.  In those terms, a standard assumption to prevent this behavior is that the drift should have a cancellation property with respect to the underlying invariant measure governing the diffusion, and that is the reason for the assumption appearing in Assumption \ref{assume:ABCompatible}.  As pointed out by one of the anonymous referees, we indicate that in fact, the special cancellation condition in Assumption \ref{assume:ABCompatible} ensures that (\ref{eqIntro:MainEpScale})-- with the help of the invariant measure-- can be written as a divergence form equation, with possibly a \emph{non-symmetric} diffusion matrix (which has been used in at least a few previous works, such as \cite{AvellanedaLin-1989CompactnessHomogNonDivEqCPAM}, \cite{BeLiPa-78}).  We will provide some brief explanation on this fact in Appendix \ref{sec:DivergenceFormEq}.
 Our main result is:
\begin{thm}\label{thm:Main}
	Assume that $g:\real^{d+1}\to\real$, $B:\real^{d+1}\to\real^{d+1}$, $A:\real^{d+1}\to\mathcal{S}(d+1)$ are all $\integer^{d+1}$ periodic functions; $A\in C^{\gam}(\real^{d+1})$ is uniformly elliptic, $B\in C^\gam(\real^{d+1})$, $g\in C^\gam(\real^{d+1})$ for some $\gam\in(0,1)$;  that $n$ is an irrational direction; and the standard compatibility condition between $A$ and $B$, that is that there is a unique invariant measure, $m$, corresponding to $\Tr(A(y)D^2u)+B(y)\cdot\grad u$ and 
\[
\int_{[0,1]^{d+1}}m(y)B_j(y)dy=0\ \ \text{for each}\ j;
\]
all of which appear in detail as Assumptions \ref{assume:A}-\ref{assume:ABCompatible}.   There exists a unique constant matrix, $\bar A$, and a unique constant, $\bar g$, such that $u^\ep\to \bar u$, uniformly in $\overline{\Sigma^1}$, where $\bar u$ is the unique affine function that solves (\ref{eqIntro:MainEffective}).
\end{thm}

\begin{rem}
When $A\in C^{1,\gam}$ and $B=\Div(A)$, Theorem \ref{thm:Main} includes the case of solving $\Div(A(x/\ep)\grad u^\ep)=0$ with $\partial_n u^\ep (x)=g(x/\ep)$. The Neumann condition is oblique with respect to the operator $\Div{(A(y)\grad u)}$ as opposed to the more frequently studied \emph{co-normal} condition, $(n, A(x/\ep)\grad u)=g(x/\ep)$.  This means that variational methods are not appropriate for (\ref{eqIntro:MainEpScale}), and so we appeal to those based on comparison principles and other non-divergence type techniques.
\end{rem}

\begin{rem}
	In Section \ref{sec:ChoiKim} we mention how a modification of Theorem \ref{thm:Main} and its proof give a different proof of a similar result for nonlinear equations without gradient dependence that was proved in Choi-Kim \cite{ChoiKim-2014HomogNeumannJMPA}.
\end{rem}

The simple geometry of $\Sigma^1$, although particular, is now known to be the most important one for resolving solutions of (\ref{eqIntro:MainEpScale}) in more general  domains.  The issues of resolving the homogenization of (\ref{eqIntro:MainEpScale}) for all irrational directions, and proving homogenization in general domains are separate ones.  In $\Sigma^1$, (\ref{eqIntro:MainEpScale}) can be thought of as a fundamental corrector problem, whereas studying (\ref{eqIntro:MainEpScale}) in general domains relies upon how the effective normal condition depends upon the normal direction, $n$, and is of a different nature than the corrector problem.

The \emph{interior} homogenization of (\ref{eqIntro:MainEpScale}) with regular boundary conditions that don't depend on $\ep$ is well understood, and $\bar A$ is determined by the same analysis that does not see the oscillatory boundary condition (see Section \ref{sec:Perturbed}).  The new part of Theorem \ref{thm:Main} is of course the ability to also have the  oscillatory condition, $\partial_n u^\ep = g(x/\ep)$.

Nearly all of our analysis that supports Theorem \ref{thm:Main} is contained in an auxiliary result about the solutions of an induced integro-differential homogenization problem-- via the Dirichlet to Neumann mapping for (\ref{eqIntro:MainEpScale}).  Here we set up and state this result, and it is expanded upon in the remainder of the work.

The Dirichlet-to-Neumann mapping (we will call it the ``D-to-N'') for (\ref{eqIntro:MainEffective}) is the mapping that uses the unique (classical) solution, $W^\ep_u$, of the equation
\begin{align*}
	\begin{cases}
		\Tr(A(\frac{x}{\ep})D^2W^\ep_u) + \frac{1}{\ep}B(\frac{x}{\ep})\cdot\grad W^\ep_u = 0\ \ &\text{in}\ \Sigma^1\\
		W^\ep_u = 0\ \ &\text{on}\ \Sigma_1\\
		W^\ep_u = u\ \ &\text{on}\ \Sigma_0,
	\end{cases}
\end{align*}
and is defined as
\begin{align*}
	\I^1: C^{1,\gam}(\Sigma_0)\to C^{\tilde\gam}(\Sigma_0)
\end{align*}
via
\begin{align*}
	\I^1(u,x) = \partial_n W^\ep_u(x),
\end{align*}
where $u$ is given, and $W^\ep_u$ is the unique solution of (\ref{eqIntro:MainEpScale}).
We want to stress that in fact $\I^1$ is an operator that has oscillatory dependence like $x/\ep$, just as the original equation.
Due to the fact that $\I^1$ is an operator with the global comparison principle, the results of Courr\`ege \cite{Courrege-1965formePrincipeMaximum} show that $\I^1$ is an integro-differential operator on functions on $\Sigma_0$, and it must have the form
\begin{align}\label{eqIntro:Courrege}
	\I^1(u,x) &= c^\ep(x)u(x) + b^\ep(x)\cdot\grad u(x)\nonumber  \\
	&\ \ \ \ + \int_{\Sigma_0}\left(u(x+h)-u(x)-\Indicator_{B_1}(h)\grad u(x)\cdot h\right)\mu^\ep(x,dh),
\end{align}
for some $c^\ep$ and $b^\ep$, and $\mu^\ep$ are appropriate L\'evy measures.  The actual quantitative properties of $c^\ep$, $b^\ep$, and $\mu^\ep$, such as regularity and L\'evy bounds, are somewhat tricky to ascertain, and so we provide analysis that does not use their exact structure.

	A recasting of equation (\ref{eqIntro:MainEpScale}) into the setting of $\I^1$ shows that $u^\ep|_{\Sigma_0}$ is the unique solution of the equation,
    \begin{align}\label{eqIntro:NonlocalHomog}
    \I^1(u^\ep|_{\Sigma_0},x)=g(\frac{x}{\ep}).
    \end{align}
	The integro-differential formula, combined with the fact that we have periodic ingredients restricted to the irrational hyperplane, $\Sigma_0$, puts (\ref{eqIntro:NonlocalHomog}) into the realm of almost periodic integro-differential homogenization of a uniformly elliptic operator on $\Sigma_0$, which we can view in a similar light to Ishii's work in \cite{Ishii-2000AlmostPeriodicHJHomog}.  Thus, our work is entirely focused on the new integro-diffrerential equation, (\ref{eqIntro:NonlocalHomog}), which is indeed an \emph{almost periodic homogenization problem}.  The main step in obtaining Theorem \ref{thm:Main} is the following:
\begin{thm}\label{thm:LimitOfBoundaryValues}
	There exists a unique constant, $\bar c$, such that $u^\ep|_{\Sigma_0}\to \bar c$ uniformly in $\Sigma_0$. 
\end{thm}

\subsection{Heuristics}

Our treatment of Theorem \ref{thm:Main} follows that of our earlier result in \cite{GuSc-2014NeumannHomogPart1DCDS-A}, and as just mentioned above, the main idea is to recast (\ref{eqIntro:MainEpScale}) as a global, interior, almost periodic integro-differential homogenization problem that takes place on $\Sigma_0$ only, given by (\ref{eqIntro:NonlocalHomog}).  The reader can skip to Section \ref{sec:MainIdeas} to see the derivation of the integro-differential homogenization and how Theorem \ref{thm:LimitOfBoundaryValues} will be proved.  

The hope is that $\I^1$ turns out to be a uniformly elliptic operator, and (\ref{eqIntro:NonlocalHomog}) becomes its own homogenization problem; hence it was natural to model our analysis on the almost periodic homogenization of Ishii for Hamilton-Jacobi equations in \cite{Ishii-2000AlmostPeriodicHJHomog}.  Furthermore, since the Courr\`ege result ensures that $\I^1$ is has an integro-differential representation like (\ref{eqIntro:Courrege}), we drew inspiration from one of the author's previous works on integro-differential homogenization in \cite{Schw-10Per}.   To make this program happen, we need to establish the following main steps: 
\begin{itemize}
	\item Show that $\I^1$ is uniformly elliptic in an appropriate sense (basically that adding appropriate bump functions can strictly increase/decrease the values of the operator), that $\I^1$ has a comparison principle between sub and super solutions, and detail how the size of the right hand side of equations involving $\I^1$ effects the size of solutions.  (This is contained in Section \ref{sec:StructureOfIep}.)
	\item Determine the relationship between $\I^1(u+c)$ and $\I^1(u)$ when $c$ is a constant. (Also in Section \ref{sec:StructureOfIep}.)
	\item Most importantly, establish that $\I^1$ is an almost periodic operator, in an appropriate sense. (This is in Section \ref{sec:AlmostPeriodicI}.)
	\item Prove that $u^\ep|_{\Sigma_0}$ has subsequential limits that can only be constants, and prove that any such constant is unique.  (This is the direct analog with Ishii \cite{Ishii-2000AlmostPeriodicHJHomog}-- in particular showing approximate correctors-- and is in some sense a direct consequence of the combination of almost periodicity and uniform ellipticity.  It appears in Sections \ref{sec:NonlocalIshii}, \ref{sec:UniqueLimitConstant}.)
\end{itemize}

As the inquisitive reader will see, in \cite{GuSc-2014NeumannHomogPart1DCDS-A} the above points were already established for purely second order equations that are translation invariant in $\Sigma_0$ in (\ref{eqIntro:MainEpScale}), with only oscillatory boundary terms.  The difficulty in this current work was finding a way to show that the presence of the term $(1/\ep)B(\frac{x}{\ep})\cdot\grad u^\ep$ in (\ref{eqIntro:MainEpScale}) did not prevent the validity of any of these steps.  In fact, the influence of this term caused significant technical difficulties in nearly all of the above steps, and it is the main reason why the analysis here is more involved than our earlier one on simpler equations in \cite{GuSc-2014NeumannHomogPart1DCDS-A}.  It turns out that we learned (after the fact) that this difficulty is directly related to the issue of finding global solutions, and/or sub/super solutions on half spaces, to equations that are an unscaled version of (\ref{eqIntro:MainEpScale}) and have appropriate uniform growth/decay at infinity (there is a small discussion on how this effects boundary homogenization by Barles-Mironescu in \cite{BarlesMironescu-2013HomogOscDirichletAsympAnal}).  This was a rather subtle issue, and the only reason our program works is because the equation in the interior in (\ref{eqIntro:MainEpScale}) indeed enjoys its own homogenization result.  This allowed us to leverage the nice behavior of solutions (hyperplanes) to the effective equation (\ref{eqIntro:MainEffective}) to show that the behavior of the $\ep$-level operators, $\I^1$, were still well behaved and the outlined program could continue.  These technical difficulties appear more or less in Section \ref{sec:HabemusBarrieram} and their results are propagated through Sections \ref{sec:StructureOfIep} and \ref{sec:HomogMainArgs}.  We want to stress that in the case that the operator in the interior of (\ref{eqIntro:MainEpScale}) is purely second order (but also allowing for $x/\ep$ dependence), these technical difficulties disappear completely due to the fact that hyperplanes are solutions to the interior equation, for all $\ep>0$.

\subsection{Comments on the title and our method of analysis}

We would like to point out that using the integro-differential framework to guide our analysis is, in principle, purely a choice of style and familiarity.  As mentioned in the heuristics, we were driven by the choice to pursue as direct an analog as possible to the Ishii result on almost periodic homogenization.  This is because we felt the nonlocal framework, by allowing to study a global elliptic problem without a boundary term, served as a nice way to package the analysis to keep track of how closely it parallels the homogenization theory for other equations with the comparison principle, e.g. uniformly elliptic equations, Hamilton-Jacobi equations, and integro-differential equations.  

That being said, we stress that it was simply a style choice.  For the curious reader, one can go back to the simpler setting when (\ref{eqIntro:MainEpScale}) is purely second order and is translation invariant, containing only the oscillatory term $\partial_n u^\ep=g(x/\ep)$.  In this simpler case, one could argue that there is a close link between the main steps in our previous intergro-differential styled approach in \cite{GuSc-2014NeumannHomogPart1DCDS-A} and the earlier approach established by Choi-Kim-Lee in \cite{ChKiLe-2012HomogNeumann} (which also has nearly the same steps by Choi-Kim \cite{ChoiKim-2014HomogNeumannJMPA}).  Thus, it is reasonable to believe that modifications of the methods in Choi-Kim \cite{ChoiKim-2014HomogNeumannJMPA} could yield the same results as this paper.  Since seemingly nothing had been proved about the homogenization of (\ref{eqIntro:MainEpScale}) in the \emph{irrational} case in a way that did not substantially link the choice of normal vector, $n$, with the drift, $B$, we chose to continue the pursuit along the lines of our previous work.  Given the analogy of our method and those of \cite{ChoiKim-2014HomogNeumannJMPA}, it is likely that the same difficulties would be encountered and similar tricks would be needed to overcome them.  It is probably unlikely that these modifications would be any easier in the case of the approach in \cite{ChoiKim-2014HomogNeumannJMPA}.

\section{Background and Assumptions}\label{sec:BackgrounAndAssumptions}

\subsection{Background}

We note that despite the many existing results about boundary homogenization (some listed below) to the best of our knowledge, there are none that can treat (\ref{eqIntro:MainEpScale}) when $n$ is irrational and $B$ does not have a restriction based on $n$.   The field of homogenization of elliptic, parabolic, and Hamilton-Jacobi equations in the periodic, almost periodic, and random settings is by now an enormous field with contributions from many authors.  The background for general homogenization is not presented here, and we simply try to give the references for oscillatory \emph{boundary} problems.  The interested reader can consult the books of Bensoussan-Lions-Papanicolaou \cite{BeLiPa-78} and Jikov-Kozlov-Oleinik \cite{Jiko-94} for an introduction to the subject and the survey of Engquist-Souganidis \cite{EnSo-08} for a somewhat current list of results.  Since \cite{EnSo-08}, there has been a large amount of activity, especially in the case of random coefficients and rates of convergence, which is not represented in \cite{EnSo-08}.  

\textbf{Two origins of (\ref{eqIntro:MainEpScale}).}  An equation with oscillatory Neumann data, as in (\ref{eqIntro:MainEpScale}), naturally arises in the study of a diffusion in an oscillatory environment with a prescribed reflection at the boundary which accrues a running cost measured in its local time, governed by an amount dictated by $g(x/\ep)$ (see e.g. \cite[Sec. 8]{SatoUeno-1965MultiDimDiffBoundryMarkov} and/or \cite[Chp. IV, Sec. 7]{IkedaWatanabe-1981SDE}).  Studying this equation with a regular boundary condition (no $x/\ep$ dependence) in the periodic setting goes back at least to Freidlin \cite{Freidlin-1964DirichletProblemPeriodicCoefficient} and Bensoussan-Lions-Papanicolaou \cite{BeLiPa-78}.  Probabilistically, the oscillatory running cost of the reflection is a natural consideration, for example in Tanaka \cite{Tanaka-1984HomogDiffBoundary}.  Another place where the condition $\partial_n u^\ep=g(x/\ep)$ arises is in the study of boundary layers for (\ref{eqIntro:MainEpScale}) with \emph{regular} boundary data.  In this case, when one uses a corrector to expand the $\ep$ behavior of $u^\ep$ around the smooth, effective solution, $\bar u$, they must investigate $u^\ep - \bar u - \ep u^1(x/\ep)-\ep^2 u^2(x/\ep)$, where vaguely $u^1$ and $u^2$ are first and second correctors.  However, this function no longer satisfies the original boundary conditions of $u^\ep$ and $\bar u$, and in general it will be a function of $x/\ep$ on $\Sigma_0$.  Hence, one is forced to consider oscillatory terms on $\Sigma_0$, as in (\ref{eqIntro:MainEpScale}), where the term of dominating order would be $\partial_n v^\ep(x)=\partial_nu^1(x/\ep)$, with $v^\ep$ a special choice for expanding the boundary layer.  For the Dirichlet problem, the appearance of the oscillatory boundary term due to the use of a corrector in the boundary layer expansion was already present in \cite[Chp. 3, Sec. 5]{BeLiPa-78} and Avellaneda-Lin \cite[Sec. 3.2]{AvellanedaLin-1987CompactnessHomogDivEqCPAM}, although in \cite{BeLiPa-78} the precise boundary behavior was not under consideration, and the oscillatory terms were not further studied because they were of lower order.

\textbf{The rational case.} The earliest results for linear non-divergence equations with oscillatory oblique Neumann data appear to go back to Tanaka \cite{Tanaka-1984HomogDiffBoundary}, where the boundary contains a subspace of the $\integer^{d+1}$ lattice, and the Neumann data is periodic with respect to this subspace.  Next came the works of Arisawa \cite{Arisawa-2003HomogNeumannAIHP} and subsequently Barles-Da Lio-Lions-Souganidis \cite{BaDaLiSo-2008ErgodicProbHomogNeumannIUMJ}.  These works treat the oscillatory Neumann (and more general) boundary conditions and fully nonlinear equations in a situation where effectively the hyperplane, $\Sigma_0$, would share a periodic sub-lattice with $\integer^{d+1}$.  Basically, the corresponding results in those papers would match up with a choice of $n$ as a rational vector in (\ref{eqIntro:MainEpScale}).   Both \cite{BaDaLiSo-2008ErgodicProbHomogNeumannIUMJ} and \cite{Tanaka-1984HomogDiffBoundary} include the singular gradient term that appears in (\ref{eqIntro:MainEpScale}).  This common periodicity between the boundary condition and the interior equation seems to be crucial in those works, as they solve coupled corrector (or ``ergodic'') problems in which the corrector must solve both the interior and boundary equations simultaneously.  In this same shared interior/boundary periodicity set-up, with both cases of rational and irrational choices of $n$, some problems of Dirichlet homogenization were studied in Barles-Mironescu \cite{BarlesMironescu-2013HomogOscDirichletAsympAnal}.  Our method, as outlined in Section \ref{sec:MainIdeas} effectively de-couples these two equations and searches for an effective limit on the boundary of $\Sigma^1$ by itself, which subsequently determines $\bar g$ a posteriori.  This vaguely means the we in some sense only search for approximate correctors to \emph{one} equation, the boundary integro-differential equation, instead of both the interior and boundary simultaneously.  We note that one could view the approach of Choi-Kim in \cite{ChoiKim-2014HomogNeumannJMPA} to also have the same de-coupling process for the cell problem as the limit is focused solely on the boundary behavior (but the interior problem comes back into the analysis subsequently when the continuity of the effective boundary condition is studied).

\textbf{The co-normal case.}  
In the case of divergence equations, with an oscillatory \emph{co-normal} condition, $(n,A(x/\ep)\grad u^\ep)=g(x/\ep)$, (\ref{eqIntro:MainEpScale}) was treated in the classical work of Bensoussan-Lions-Papanicolaou \cite[Chp. 1 Sec. 7.3]{BeLiPa-78}.  It in fact follows from the divergence structure of the equation, exactly as in the same way as the problem with regular boundary conditions.  There, they identify the difference between the rational and irrational cases of $n$, in domains that are more general than $\Sigma^1$.  The problem becomes more challenging when one wishes to know more precise information about the boundary layer of $u^\ep$ near $\Sigma_0$, and one seeks a rate of convergence towards the effective solution.  In that case, the analysis is much more delicate, and some works on this boundary homogenization structure in the Dirichlet and \emph{co-normal} Neumann setting for elliptic \emph{systems} include Avellaneda-Lin \cite{AvellanedaLin-1987CompactnessHomogDivEqCPAM} and Kenig-Lin-Shen  \cite{KenigLinShen2013HomogEllipticSystNeumannCondJAMS}, \cite{KenigLinShen2014-PeriodicHomogGreenAndNeumannFunctionsCPAM}, where the boundary layers are studied for equations with regular boundary conditions.  In the realm of studying \emph{oscillatory} boundary conditions, there is the work of G\'erard-Varet-Masmoudi \cite{GerardVaretMasmoudi-2012HomogBoundarLayerActa}, followed by that of Armstrong-Kuusi-Mourrat-Prange \cite{ArmstrongKuusiMourratPrange-2016QuantitativeBoundaryLayersArXiv} and Shen-Zhuge \cite{ShenZhuge-2016BoundaryLayerNeumannArXiv}.  The works \cite{ArmstrongKuusiMourratPrange-2016QuantitativeBoundaryLayersArXiv} and \cite{ShenZhuge-2016BoundaryLayerNeumannArXiv} both reduce their analysis to that of integrals on the boundary, $\partial\Om$.  \cite{ShenZhuge-2016BoundaryLayerNeumannArXiv} is close to our work here, as they treat the oscillatory Neumann problem with a \emph{co-normal} condition; it is interesting to notice a comparison/contrast with our work in the sense that both are focused on boundary analysis, yet in \cite{ShenZhuge-2016BoundaryLayerNeumannArXiv} the important object is a potential operator with an oscillatory integral (via the Neumann function), whereas in our work we study an integro-differential operator on the boundary.

\textbf{The gradient independent case.}  Recently, the boundary homogenization for both cases of oscillatory Dirichlet and Neumann conditions with fully nonlinear uniformly elliptic equations in the interior (with $x/\ep$ dependence) has been solved in many situations, including some general domains.  All of the following works treat problems with purely second order equations in the interior, hence the presence of the term $(1/\ep)B(x/\ep)$ in (\ref{eqIntro:MainEpScale}) is one aspect that separates this work from the following (also, they treat nonlinear equations, while we work with linear ones).  The work of Choi-Kim-Lee \cite{ChKiLe-2012HomogNeumann} solved the problem in a strip domain, as $\Sigma^1$, with a translation invariant nonlinear equation and oscillatory Neumann condition (subsequently the authors, in \cite{GuSc-2014NeumannHomogPart1DCDS-A}, gave a different proof along the lines presented in this work).  Then Choi-Kim \cite{ChoiKim-2014HomogNeumannJMPA} treated the problem in more general domains with also an oscillatory equation, $F(D^2u^\ep,x/\ep)=0$, in the interior.  In \cite{ChoiKim-2014HomogNeumannJMPA} it became clear that equations in the simple domain, like (\ref{eqIntro:MainEpScale}), can serve as a sort of corrector problem for the boundary behavior of $u^\ep$ in more general domains, and they also noted there are two separate questions that must be addressed: the solvability of the corrector problem, such as Theorem \ref{thm:Main}, and the continuity properties of $\bar g$ as a function of the normal $n(x)$, for $x$ in the boundary.  The related Dirichlet homogenization in general domains was solved by Feldman \cite{Feldman-2014HomogDirichletGeneralDomainJMPA}, where crucially it was noted that there may possibly be instances in which the effective boundary condition will be discontinuous, but nonetheless one can control the size of the set of discontinuities in a way that still guarantees the effective Dirichlet problem admits unique solutions.  The issue of continuity / discontinuity of the effective boundary conditions has been further studied by Feldman-Kim \cite{FeldmanKim-2015ContDiscontBoundaryLayerArXiv}.   Finally, the homogenization of random oscillatory boundary data with a translation invariant uniformly elliptic equation in the interior was recently obtained for some types of random environments by Feldman-Kim-Souganidis \cite{FeldmanKimSouganidis-2014RandomBoundaryHomogJMPA}.

As one final note, we mention that in Bal-Jing \cite{BalJing-2011CorrectorSingularGreenRandomBoundaryCommMathSci}, a Dirichlet-to-Neumann mapping was utilized-- as we do here-- for an equation that is translation invariant in the interior and has a Robin boundary condition with a random, stationary ergodic term.

\subsection{Assumptions}\label{sec:Assume}

\begin{assume}\label{assume:A}
$A: \real^{d+1}\to \mathcal{S}(d+1)$ (symmetric $(d+1) \times (d+1)$ matrices) and for some $\lam<\Lam$, for all $x$, $\lam\Id\leq A(x)\leq\Lam\Id$.  $A\in C^\gam(\real^{d+1})$ and is $\integer^{d+1}$ periodic.
\end{assume}

\begin{assume}\label{assume:B}
	$B:\real^{d+1}\to\real^{d+1}$, $B\in C^\gam(\real^{d+1})$, and $B$ is $\integer^{d+1}$ periodic.
\end{assume}

\begin{rem}\label{rem:InvariantMeasureExistsInL2}
	Assume that $Lu(y)=A_{ij}(y)u_{y_iy_j}(y) + B(y)\cdot\grad u(y)$ and that $L^*u=(A_{ij}(\cdot)u)_{y_iy_j}-\Div(uB(\cdot))$ is the formal adjoint of $L$.   Under Assumptions \ref{assume:A} and \ref{assume:B}, there exists a unique invariant measure, $m\in L^2([0,1]^{d+1})$, which is periodic and solves
	\begin{align*}
		L^*m=0,
	\end{align*}
	subject to $m>0$ and 
	\begin{align*}
		\int_{[0,1]^{d+1}} m(y) dy = 1.
	\end{align*}
	(We mean that
	\begin{align*}
		\int m(y)L\phi(y)dy=0
	\end{align*}
	for all periodic $\phi\in C^2\intersect L^\infty$.)  Some details of $m$ are briefly expanded upon in Appendix \ref{sec:InvariantMeasure}.
\end{rem}

\begin{assume}\label{assume:ABCompatible}
	$B$ must satisfy the compatibility (or ``centering'') condition that
	\begin{align*}
		\int_{[0,1]^{d+1}}B(y)m(y)dy=0.
	\end{align*}
	It prevents trajectories of the related rescaled stochastic process from ``escaping to infinity'', and gives compactness of the measures for the law of the processes.
\end{assume}

The main point of our method is that we study an almost periodic homogenization problem on $\Sigma_0$.  We use the following definition for almost periodic functions.

\begin{DEF}\label{def:AlmostPeriodic}
	$f:\Sigma_0\to\real$ is almost periodic if for any $\delta>0$, the set of $\del$-almost periods of $f$, 
		\[
		E_\delta := \{\tau\in\Sigma_0\ :\ \sup_{x\in\Sigma_0}\abs{f(x+\tau)-f(x)}< \del\},
		\]
		satisfies the following property: there exists a compact set, $K\subset \Sigma_0$, such that 
		\[
		(z+K)\intersect E_\delta \not= \emptyset\ \textnormal{for all}\ z\in\Sigma_0.
		\]
\end{DEF}

\begin{rem}\cite[Proposition 1.2]{Shubin-1978AlmostPeriodic}
	Two other equivalent and classical formulations are that $f:\Sigma_0\to\real$ is \emph{almost periodic} if
	\begin{itemize}
		\item[(i)] $f$ can be uniformly approximated on $\Sigma_0$ by trigonometric polynomials.

		\item[(ii)] The set
		$\displaystyle
			\left\{ f(\cdot+z)\ :\ z\in\Sigma_0   \right\}
		$
		is precompact in the space $L^\infty(\Sigma_0)$.
	\end{itemize}
\end{rem}

\subsection{Notation}\label{sec:Notation}

\begin{itemize}
	\item We will assume that $x\in\real^{d+1}$ is written in the coordinates relative to $\Sigma_0$, that is $x=(\hat x, x_{d+1})$, where $\hat x\in\Sigma_0$ and $x_{d+1}\in\textnormal{span}(n)$.
	\item $\displaystyle \Sigma^r = \{x\in\real^{d+1}\ :\ 0<x\cdot n<r\}$
	\item $\displaystyle \Sigma_r = \{x\in\real^{d+1}\ :\ x\cdot n=r\}$
	\item $\displaystyle \Sigma_0 = \{x\in\real^{d+1}\ :\ x\cdot n=0\}$
	\item typically, $x$, is the original macroscopic variable, as in (\ref{eqIntro:MainEpScale})
	\item typically, $y$, is the microscopic variable, as in (\ref{eqMainIdea:BulkDirichletAtMicroScale}), which results from a rescaling such as $v(y) = (1/\ep)u^\ep(\ep y)$.
	\item $\displaystyle Lu(y)=A_{ij}(y)u_{y_iy_j}(y) + B(y)\cdot\grad u(y)$
	\item $\displaystyle B_R\subset \real^{d+1}$, $\displaystyle B_R'\subset \Sigma_0$, $B_R^+=B_R\intersect \Sigma^r$ for $r>R$,
	\item $\M^+$ and $M^-$ are the Pucci operators from the second order elliptic theory (see \cite[Chp 2]{CaCa-95}), defined as 
	\[
	\M^+(D^2u)=\sup_{\lam\Id\leq P\leq \Lam\Id}\left( \Tr(PD^2u) \right)
	\ \ \text{and}\ \ 
	\M^-(D^2u)=\inf_{\lam\Id\leq P\leq \Lam\Id}\left( \Tr(PD^2u) \right).
	\]
\end{itemize}

\section{Main Ideas of The Proof}\label{sec:MainIdeas}

In this section, we give the sketch of the proof, without explicit justifications.  The details appear in Sections \ref{sec:HabemusBarrieram}, \ref{sec:StructureOfIep}, \ref{sec:HomogMainArgs}.  As briefly mentioned in the introduction, our strategy in this approach centers on the Dirichlet-to-Neumann (D-to-N) operator.  In fact, we will use two such operators at two different scales.  First is the original scale of the macroscopic variables, with $W^\ep_u$ the unique (classical) solution of
\begin{align}\label{eqMainIdea:BulkDirichletAtMacroScale}
	\begin{cases}
		\Tr(A(\frac{x}{\ep})D^2W^\ep_u) + \frac{1}{\ep}B(\frac{x}{\ep})\cdot\grad W^\ep_u = 0\ \ &\text{in}\ \Sigma^1\\
		W^\ep_u = 0\ \ &\text{on}\ \Sigma_1\\
		W^\ep_u = u\ \ &\text{on}\ \Sigma_0,
	\end{cases}
\end{align}
and we define
\begin{align*}
	\I^1: C^{1,\gam}(\Sigma_0)\to C^{\tilde\gam}(\Sigma_0)
\end{align*}
via
\begin{align*}
	\I^1(u,x) = \partial_n W^\ep_u(x).
\end{align*}
Similarly, we will also need the D-to-N mapping for the equation in microscopic variables.  This changes the domain of definition of the equation, to the larger set, $\Sigma^{1/\ep}$:
\begin{align}\label{eqMainIdea:BulkDirichletAtMicroScale}
	\begin{cases}
		\Tr(A(y)D^2V^\ep_u) + B(y)\cdot\grad V^\ep_u = 0\ \ &\text{in}\ \Sigma^{1/\ep}\\
		V^\ep_u = 0\ \ &\text{on}\ \Sigma_{1/\ep}\\
		V^\ep_u = u\ \ &\text{on}\ \Sigma_0.
	\end{cases}
\end{align}
We define
\begin{align*}
	\I^{1/\ep}: C^{1,\gam}(\Sigma_0)\to C^{\tilde\gam}(\Sigma_0)
\end{align*}
via
\begin{align}\label{eqMainIdea:IinMicroVars}
	\I^{1/\ep}(u,y) = \partial_n V^\ep_u(y).
\end{align}
We note that $u\in C^{1,\gam}(\Sigma_0)$, implies that $V^\ep_u, W^\ep_u\in C^{1,\gam}(\overline{\Sigma^1})$, which makes $\I^1$ and $\I^{1/\ep}$ well defined.

Many subsequent arguments will invoke various regularity estimates involving $u^\ep$ and $g(\cdot/\ep)$.  However, these will not be useful unless it can first be shown that $\norm{u^\ep}_{L^\infty}\leq C$, independently from $\ep$.  Without assumptions on $B$, such as the ``centering condition'' in Assumption \ref{assume:B}, these estimates are false.  However, it turns out that one can use the interior homogenization for an auxiliary barrier equation to derive the needed control on $\norm{u^\ep}_{L^\infty}$. 
Thus, since $g\in C^{\gam}(\Sigma_0)$, standard regularity theory (see Proposition \ref{prop:C1GamRegNeumannProblem} as well as the $L^\infty$ estimate in Section \ref{sec:LInfinityEstimate}) indicates that $u^\ep\in C^{1,\gam}(\overline{\Sigma^1})$ with $[u^\ep]_{C^\gam}$ independent of $\ep$ and $[\grad u^\ep]_{C^\gam}$ depending on $\ep$.  Thus, this says that up to subsequences, $u^\ep$ will have local uniform limits in $\overline{\Sigma^1}$.  The main question is, can we characterize all possible such limits?

For the moment, \emph{assume} that there is a \emph{unique} $\bar v\in C^{\gam}(\Sigma_0)$ such that $u^\ep|_{\Sigma_0}\to \bar v$.  Letting $\bar u$ be any possible local uniform limit of $u^\ep$ in $\Sigma^1$, we know by the perturbed test function method for viscosity solutions of (\ref{eqIntro:MainEpScale}), that $\bar u$ must solve (see Proposition \ref{prop:PTFM})
\begin{align*}
	\Tr(\Bar A D^2\bar u)=0\ \text{in}\ \Sigma^1,
\end{align*}
where $\bar A$ is the unique constant matrix defined (\ref{eqPerturb:ABarDef}).  Hence, by the global $C^\gam$ continuity of $\bar u$ and the \emph{assumed} convergence of $u^\ep|_{\Sigma_0}\to \bar v$ for a unique $\bar v$, we see that $\bar u$ is the unique solution of 
\begin{align*}
	\begin{cases}
		\Tr(\Bar A D^2\bar u)=0\ &\text{in}\ \Sigma^1,\\
		\bar u = 0\ &\text{on}\ \Sigma_1\\
		\bar u =\bar v\ &\text{on}\ \Sigma_0. 
	\end{cases}
\end{align*} 
Thus, assuming $\bar v\in C^{1,\gam}(\Sigma_0)$, then $\bar u\in C^{1,\gam}(\overline{\Sigma^1})$.  We obtain, a posteriori that the unique effective Neumann condition is then
\begin{align*}
	\bar g = \partial_n \bar u.
\end{align*}
(We note that in the flat geometry of $\Sigma^1$, it will be that $\bar g$ is a constant.)

Now, why will it be true that there is a unique $\bar v$, such that $u^\ep|_{\Sigma_0}\to \bar v$? As mentioned in the introduction, answering this question is the main result of our work, which we stated as Theorem \ref{thm:LimitOfBoundaryValues}.  Because $u^\ep\in C^{1,\gam}$, the equation (\ref{eqIntro:MainEpScale}) implies that $u^\ep|_{\Sigma_0}$ is the unique classical solution of
\begin{align}\label{eqMainIdea:IntDiffHomog}
	\I^1(u^\ep|_{\Sigma_0},x) = g(\frac{x}{\ep}).
\end{align}

We will show that the irrationality of $n$, with the periodicity of $A$, $B$, $g$, imply that $\I^1$ is an almost periodic operator on $\Sigma_0$ and $g$ is an almost periodic function.  We will then invoke standard ideas from Hamilton-Jacobi homogenization theory, like Ishii \cite{Ishii-2000AlmostPeriodicHJHomog}, to show that the limit of $u^\ep|_{\Sigma_0}$ is a unique constant.   We note that an interesting aspect of our work is the need to establish estimates on $\norm{u^\ep}_{L^\infty}$ and $\norm{\partial_n u^\ep}_{L^\infty}$, which begins in the next section.

\section{The Construction of Barriers and a bound for $\norm{u^\ep}_{L^{\infty}}$}\label{sec:HabemusBarrieram}

This section is dedicated to a deceptively simple result; the construction of barriers, contained in Proposition \ref{prop:BarrierNonTrivNormalDeriv}, below.  The interesting part is that when $B\equiv 0$ in (\ref{eqIntro:MainEpScale}), even for fully nonlinear equations, this proposition is trivial from the observation that affine functions solve the equation in $\Sigma^1$ (more on this in Section \ref{sec:ChoiKim}).  Even more intriguing is that the outcome of Proposition \ref{prop:BarrierNonTrivNormalDeriv} is false, unless homogenization occurs in the interior of $\Sigma^1$.   In fact, as the reader will see, the result utilizes the convenient result that the interior homogenization of the regular Dirichlet problem enjoys a \emph{global} rate up to the boundary of $\Sigma^1$; this was proved in the original work of Bensoussan-Lions-Papanicolaou for the non-divergence equation, (\ref{eqIntro:MainEpScale}), \cite[Chp. 3, Sec 5, Thm. 5.1]{BeLiPa-78} and in the divergence setting by Avellaneda-Lin \cite{AvellanedaLin-1987CompactnessHomogDivEqCPAM}.  We are not certain if such heavy machinery is necessary, but it sufficed for the present investigation.

We then use the barriers to obtain the crucial estimate that is needed to begin the homogenization procedure, namely that
\begin{align*}
	\norm{u^\ep}_{L^\infty(\Sigma^1)}\leq C,\ \text{uniformly in}\ \ep.
\end{align*}
It is interesting in this case, that the ``centering'' assumption on $B$ and the homogenization for regular boundary conditions are used to get and estimate for $\norm{u^\ep}_{L^\infty}$.

\subsection{The barrier}

\begin{prop}[Barrier]\label{prop:BarrierNonTrivNormalDeriv}
	There exist (universal) constants, $c_1$ and $c_2$, $0<c_1<c_2$, such that
	\begin{align}\label{eqBarr:NormalDerivEst}
		-c_2 \leq \partial_n \phi^\ep \leq -c_1<0,
	\end{align}
	where $\phi^\ep$ is the unique solution of the problem
	\begin{align}\label{eqBarr:EqForPhiEp}
		\begin{cases}
		\Tr(A(y)D^2\phi^\ep(y))+B(y)\cdot\grad \phi^\ep(y)=0\ &\text{in}\ \Sigma^{1/\ep}\\
		\phi^\ep = 0\ &\text{on}\ \Sigma_{1/\ep}\\
		\phi^\ep = 1/\ep\ &\text{on}\ \Sigma_0.
		\end{cases}
	\end{align}
\end{prop}

An immediate consequence of the barrier behavior in Proposition \ref{prop:BarrierNonTrivNormalDeriv} is

\begin{prop}[Estimates for The Neumann Problem]\label{prop:NemannEstimates}
	 There exists an universal constant, $C$, so that if for some $\del\geq0$, $\eta^\ep$ solves
	\begin{align}\label{eqBarr:EqForEtaEp}
		\begin{cases}
		\Tr(A(y)D^2\eta^\ep(y))+B(y)\cdot\grad \eta^\ep(y)=0\ &\text{in}\ \Sigma^{1/\ep}\\
		\abs{\eta^\ep} \leq \del\ &\text{on}\ \Sigma_{1/\ep}\\
		\abs{\partial_n\eta^\ep} \leq \del\ &\text{on}\ \Sigma_0,
		\end{cases}
	\end{align}
	then it holds that
	\begin{align*}
		\norm{\eta^\ep}_{L^\infty(\overline{\Sigma^{1/\ep}})} \leq C\frac{\del}{\ep}.
	\end{align*}
\end{prop}

First we will prove the estimates for the Neumann problem, Proposition \ref{prop:NemannEstimates}, and then we will prove the Barrier, Proposition \ref{prop:BarrierNonTrivNormalDeriv}.

\begin{proof}[Proof of Proposition \ref{prop:NemannEstimates}]
	We will only prove one side of the bound for $\eta^\ep$, and the reverse inequality is analogous.  We begin with a note that due to the positive $1$-homogeneity of the operator, $L$, any positive multiple of $\phi^\ep$ will be a supersolution, i.e. $L\phi^\ep\leq 0$. We claim that for an appropriate $b>0$, the function,
	\begin{equation}
		\tilde \phi^\ep = b\phi^\ep + \del
	\end{equation}
	is a supersolution of (\ref{eqBarr:EqForEtaEp}).  Indeed, choosing $b>0$ large enough (w.l.o.g. $b>1$), we obtain
	\begin{align*}
		\partial_n \tilde \phi^\ep < -bc_1 < -\del \leq \partial_n\eta\ \text{on}\ \Sigma_0,
	\end{align*}
	where $c_1$ is the constant appearing in Proposition \ref{prop:BarrierNonTrivNormalDeriv}.  It is useful to note that the choice of $b$ is dictated by
	\[
	bc_1>\del,\ \ \text{so choosing}\ \ b=2\del/c_1\ \ \text{will suffice}.
	\]  
	
	Furthermore, by the construction of $\phi^\ep$,
	\begin{align*}
		\tilde \phi^\ep|_{\Sigma_{1/\ep}} = \del \geq \eta|_{\Sigma_{1/\ep}}.
	\end{align*}
	Finally, by the positive 1-homogeneity of (\ref{eqBarr:EqForPhiEp}), we have not changed the equation for $\phi^\ep$ and $\tilde \phi^\ep$.  
	
	The upper bound on $\eta^\ep$ is now immediate from the comparison of sub and super solutions of Neumann problems (Proposition \ref{prop:NeumannComparison}), which implies
	\begin{align*}
		\eta^\ep\leq \tilde \phi^\ep \leq b/\ep + \del\ \text{in}\ \overline{\Sigma^{1/\ep}},
	\end{align*}
	and by the fact that $b=2\del/c_1$, gives the desired result.
\end{proof}

\begin{proof}[Proof Proposition \ref{prop:BarrierNonTrivNormalDeriv}]

The key component for our proof of Proposition \ref{prop:BarrierNonTrivNormalDeriv} is the result in \cite[Chp. 3, Sec 5, Thm 5.1]{BeLiPa-78} regarding rates of convergence in homogenization, which shows that $\phi^\ep$ must stay within a fixed distance of order 1 of a hyperplane.   (Note, we state the rates result of \cite{BeLiPa-78} as Proposition \ref{prop:RatesHomogGlobal} and provide modifications to allow for our assumptions on $A$ and $B$, which do not require as much regularity as the presentation in \cite{BeLiPa-78}.)

Let us define the rescaling of $\phi^\ep$ as
\begin{align*}
	\rho^\ep(x) = \ep\phi^\ep(\frac{x}{\ep}),
\end{align*}
so we have
\begin{align*}
	D^2\rho^\ep(x)= \ep^{-1}D^2\phi^\ep(\frac{x}{\ep}),\ \ \grad \rho^\ep(x) = \grad\phi^\ep(\frac{x}{\ep})
\end{align*}
Hence, $\rho^\ep$ is the unique solution of 
	\begin{align}\label{eqBarr:EqForRhoEp}
		\begin{cases}
		\Tr(A(\frac{x}{\ep})D^2\rho^\ep(y))+\frac{1}{\ep}B(\frac{x}{\ep})\cdot\grad \rho^\ep(y)=0\ &\text{in}\ \Sigma^{1}\\
		\rho^\ep = 0\ &\text{on}\ \Sigma_{1}\\
		\rho^\ep = 1\ &\text{on}\ \Sigma_0.
		\end{cases}
	\end{align}
Furthermore, by standard homogenization results, e.g. Bensoussan-Lions-Papanicolaou \cite[Chp. 3, Sec 4 and 5]{BeLiPa-78}, also summarized here in Section \ref{sec:Perturbed} we know that there is a unique \emph{constant} matrix $\bar A$ so that $\rho^\ep\to\bar \rho$, and $\bar \rho$ is the unique solution of
	\begin{align}
		\begin{cases}
		\Tr(\bar A D^2\bar\rho(y))=0\ &\text{in}\ \Sigma^{1}\\
		\rho^\ep = 0\ &\text{on}\ \Sigma_{1}\\
		\rho^\ep = 1\ &\text{on}\ \Sigma_0.
		\end{cases}
	\end{align}
Since this equation is purely second order, we see that $\bar \rho(x) = 1-x_{d+1}$.  Now we use the fact that there is a global rate of convergence of $\rho^\ep$, for example in \cite[Chp. 3, Thm 5.1]{BeLiPa-78} (see also Proposition \ref{prop:RatesHomogGlobal}).  Namely, there is a universal constant, $C$, so that
\begin{align*}
	\norm{\rho^\ep-\bar \rho}_{L^\infty(\overline{\Sigma^1})}\leq C\ep.
\end{align*}
We note also, for the divergence case ($B=\Div(A)$), we could utilize the global rate from \cite[Theorem 5]{AvellanedaLin-1987CompactnessHomogDivEqCPAM}.
In the microscopic variables, this says
\begin{align}\label{eqBarr:RateInMicroVars}
	\norm{\phi^\ep - \frac{1}{\ep}\bar\rho(\ep\cdot)}_{L^\infty(\overline{\Sigma^{1/\ep}})}\leq C.
\end{align}
Finally, we note that since we are ultimately concerned with only $\partial_n \phi^\ep$, we assume that we have subtracted $1/\ep$ from both $\phi^\ep$ and $\ep^{-1}\bar \rho(\ep\cdot)$ so that
\begin{align*}
	\phi^\ep=0\ \text{on}\ \Sigma_0,\ \phi^\ep=-1/\ep\ \text{on}\ \Sigma_{1/\ep},\
	\text{and}\ \bar\rho(y)=-y_{d+1}.
\end{align*} 

These observations are now enough to build a barrier for $\phi^\ep$.  The most important trait of $\phi^\ep$ and $\bar \rho$ is that (\ref{eqBarr:RateInMicroVars}) shows that there will be a \emph{fixed} distance, $t^*$, that is independent of $\ep$ so that $\phi^\ep(y)\leq -1$ for all $y_{d+1}=t^*$.  The fact that $t^*$ is independent of $\ep$ allows to construct a good barrier.  Indeed, we simply choose $t^*=C+1$, so from (\ref{eqBarr:RateInMicroVars})
\begin{align*}
\phi^\ep(y)\leq -1\  \text{for all}\ y_{d+1}=t^*.
\end{align*}

	We will construct an upper barrier, $M$, for $\phi^\ep$. $M$ will be a super solution in $\Sigma^{t^*}$, with $M=0$ on $\Sigma_0$, and $M=-1$ on $\Sigma_{t^*}$.  We note that $M$ is constructed to be a function only of $y_{d+1}$ and $M'\leq0$ globally.  We need
	\begin{align*}
		L(M)=\Tr(A(y)D^2M)+B(y)\cdot\grad M\leq 0\ \text{in}\ \Sigma^{t^*}.
	\end{align*}
	By the uniform ellipticity of $A$,
	\begin{align*}
	 L(M)=\Tr(A(y)D^2M)+B(y)\cdot\grad M\leq \Lam M'' - \norm{B}M',
	\end{align*}
	and so we seek a solution of the form 
	\begin{align*}
		M'' - \frac{C}{\Lam} M' = 0,\ M(0)=0,\ M(t^*)=-1.
	\end{align*}
	Let us denote $C_2=(C/\Lam)$.  The good choice of $M$ will be
	\begin{align*}
		M(t) = \frac{a_0}{C_2}(e^{C_2t}-1),
	\end{align*}
	where $a_0$ is chosen to obtain $M(t^*)=-1$.  That means that 
	\begin{align*}
		\displaystyle
		a_0 = \frac{-C_2}{(e^{C_2t^*}-1)} <0,
	\end{align*}
	and without loss of generality, $C_2\geq 1$, and $t^*\geq 1$.  Thus $M$ is a super solution of $L(M)\leq 0$ in $\Sigma^{t^*}$, $\phi^\ep$ is a solution of $L(\phi^\ep)=0$ in $\Sigma^{t^*}$, $\phi^\ep=M$ on $\Sigma_0$, and $M\geq \phi^\ep$ on $\Sigma_{t^*}$.  Thus 
	\begin{align*}
		\phi^\ep\leq M\ \text{in}\ \overline{\Sigma^{t^*}},\ \text{and}\ \partial_n\phi^\ep(y)\leq M'(0) = a_0<0\ \text{for}\ y\in\Sigma_0.
	\end{align*}
	This concludes the claimed upper bound of the proof, with $c_1=a_0$.
	
	Next, we construct a lower barrier for $\phi^\ep$, which we will call $m$.  For the lower bound construction, we can simply choose $s^*=10$.  We will build $m$ so that
	\begin{align*}
		m'\leq 0\ \text{globally},
	\end{align*}
	and
	\begin{align*}
		L(m)= \Tr(A(y)D^2M)+B(y)\cdot\grad M\geq 0,\ m=0\ \text{on}\ \Sigma_0,\ m=-1/\ep\ \text{on}\ \Sigma^{s_*}
	\end{align*}
	the sign of $m'$ plus the ellipticity says
	\begin{align*}
		L(m) \geq \lam m''+ \norm{B}m'
	\end{align*}
	so we choose it concretely to solve 
	\begin{align*}
		m'' + (C/\lam)m' = 0\ \text{with}\ m(0)=0,\ \text{and}\ m(10)=-10-C.
	\end{align*}
	Hence,
	\begin{align*}
		m(t) = \frac{a_1}{C_3}(1-e^{-C_3t}),
	\end{align*}
	where $C_3=(C/\lam)$.  

 We make the choice
	\begin{align*}
		a_1 = \frac{C_2(-10-C)}{(1-e^{-C_2*10})} < 0.
	\end{align*}
	Furthermore, the rates of convergence of $\phi^\ep\to\bar\rho$ show that 
	\begin{align*}
		\phi^\ep(y)\geq -10 - C\ \text{for all}\ y\ \text{such that}\ y_{d+1}=10,
	\end{align*}
	and hence
	\begin{align*}
		m(10) = -10 - C \leq \phi^\ep(y)\ \text{for}\ y_{d+1}=10.
	\end{align*}
	So, again, we conclude that $m$ is a subsolution in $\Sigma^{s^*}$, and so $\phi^\ep\geq m$.  We note
	\begin{align*}
		m'(0) = a_1 < 0, 
	\end{align*}
	and thus the lower bound holds with $c_2=a_1$.
\end{proof}

\begin{rem}
	The proof of the previous proposition is basically the one feature that separates the proofs for the case where $L$ is independent from the gradient and the case of (\ref{eqIntro:MainEpScale}).  When the equation is independent of the gradient, Proposition \ref{prop:BarrierNonTrivNormalDeriv} is trivial because affine functions are solutions of the equation.
\end{rem}

We conclude this section with one last result that captures the limiting behavior of $\phi^\ep-1/\ep$ as $\ep\to0$.  It will be useful for proving the uniqueness of the limiting boundary condition in Theorem \ref{thm:Main}.

\begin{lem}[Limit in half-space]\label{lem:LimitInInfiniteDomain}
	Assume that $\psi^\ep$ solves
	\begin{align*}
		\begin{cases}
		\Tr(A(y)D^2\psi^\ep(y))+B(y)\cdot\grad \psi^\ep(y)=0\ &\text{in}\ \Sigma^{1/\ep}\\
		\psi^\ep = -1/\ep\ &\text{on}\ \Sigma_{1/\ep}\\
		\psi^\ep = 0\ &\text{on}\ \Sigma_0.
		\end{cases}
	\end{align*}
	Then there exists a unique $\psi^\infty$ such that $\psi^\ep\to\psi^\infty$ uniformly in $\overline{\Sigma^{1}}$ as $\ep\to0$.  Furthermore, $\psi_\infty$ is characterized as the unique solution in a particular growth class to this equation in the infinite domain.
\end{lem}

\begin{proof}
	First, we note that uniqueness of solutions to this equation in half space is not guaranteed unless there are particular sub/super solutions with uniform growth/decay at infinity (this played a role in \cite{BaDaLiSo-2008ErgodicProbHomogNeumannIUMJ} and \cite{BarlesMironescu-2013HomogOscDirichletAsympAnal}).  However, the analysis and barriers from Proposition \ref{prop:BarrierNonTrivNormalDeriv} give the desired sub/super solutions.  Thus, there is uniqueness in the class of bounded solutions of 
	\begin{align}\label{eqBarr:InfiniteDomain}
		\begin{cases}
		\Tr(A(y)D^2u(y))+B(y)\cdot\grad u(y)=0\ &\text{in}\ \Sigma^{\infty}\\
		u=u_0\ &\text{on}\ \Sigma_0,
		\end{cases}
	\end{align}
	where 
\begin{align*}
	\Sigma^\infty=\{x\in\real^{d+1}\ :\ 0<x\cdot n\}. 
\end{align*}
Thus if there are $u$ and $v$ that solve (\ref{eqBarr:InfiniteDomain}) and for a fixed constant, satisfy
	\begin{align*}
		\abs{u(y)-y_{d+1}}\leq C,\ \text{and}\ \abs{v(y)-y_{d+1}}\leq C\ \text{for all}\ y\in\Sigma^\infty,
	\end{align*}
	then $u-v$ is bounded, solves (\ref{eqBarr:InfiniteDomain}) with $u-v=0$ on $\Sigma_0$, and hence is the unique bounded solution, which is $u-v\equiv0$.

	To identify $\psi^\infty$, we note that $\psi^\ep$ does indeed have local uniform limits.  This is due to the fact that its oscillation in any ball of a fixed radius is uniformly bounded-- as seen in the proof of Proposition \ref{prop:BarrierNonTrivNormalDeriv}-- and hence enjoys a global estimate on $[\psi^\ep]_{C^\gam}$, by uniformly using the local estimate to cover $\Sigma^{1/\ep}$.  But then if $\psi^\infty$ is \emph{any} local uniform limit of $\psi^\ep$, it follows as in the proof of Proposition \ref{prop:BarrierNonTrivNormalDeriv} that 
	\begin{align*}
		\abs{\psi^\ep(y)-y_{d+1}}\leq C,
	\end{align*}
	as well as $\psi^\infty$ satisfies (\ref{eqBarr:InfiniteDomain}) by the stability of (viscosity) solutions.  Hence any possible $\psi^\infty$ will indeed satisfy
	\begin{align*}
		\abs{\psi^\infty-y_{d+1}}\leq C,
	\end{align*}
	and is thus uniquely characterized as the solution to \ref{eqBarr:InfiniteDomain} with $u_0=0$ and that is globally a bounded amount from $y_{d+1}$.
	
	Now, to conclude, we need to check that the convergence is in fact uniform in the strip $\overline{\Sigma^1}$.  We note that in $\Sigma^{1/\ep}$, $\psi^\ep-\psi^\infty$ is a bounded solution of 
	\begin{align*}
		\Tr(A(y)D^2u(y))+B(y)\cdot\grad u(y)=0\ \text{in}\ \Sigma^{1/\ep}.
	\end{align*}
	Thus by the maximim principle,  $\max_{\overline{\Sigma^{1/\ep}}}(\psi^\ep-\psi^\infty)$, occurs on the boundary of $\Sigma^{1/\ep}$.  Since $\psi^{1/\ep}-\psi^\infty=0$ on $\Sigma_0$ and  $\abs{\psi^{1/\ep}-\psi^\infty}\leq C$ on $\Sigma^{1/\ep}$, we conclude that 
	\begin{align*}
		\ep C(\phi^\ep-1/\ep)\ \text{and}\ -\ep C(\phi^\ep-1/\ep)
	\end{align*}  
	are respectively sub and super solutions of the equation and boundary data, where $\phi^\ep$ is defined in (\ref{eqBarr:EqForPhiEp}).  Hence
	\begin{align*}
		\ep C(\phi^\ep-1/\ep)\leq \psi^\ep-\psi^\infty\leq -\ep C(\phi^\ep-1/\ep)\ \text{in}\ \overline{\Sigma^{1/\ep}}.
	\end{align*}
	Since $\abs{\phi^\ep-1/\ep}\leq C$ in $\overline{\Sigma^1}$ (again, due to the arguments in the proof of Proposition \ref{prop:BarrierNonTrivNormalDeriv}), we conclude that 
	\begin{align*}
		\sup_{y\in\Sigma^1}\abs{\psi^{1/\ep}-\psi^\infty}\to 0\ \text{as}\ \ep\to0.
	\end{align*}
\end{proof}

\subsection{Uniform bound for $\norm{u^\ep}_{L^\infty}$}\label{sec:LInfinityEstimate}

We will use, in a crucial way, the $C^\gam$ and $C^{1,\gam}$ estimates for $u^\ep$, and various other related functions (see Corollary \ref{cor:uEpHolderAndGradientHolder} and Proposition \ref{prop:NemannEstimates} and \ref{prop:C1GamRegNeumannProblem}).  For these to be useful, we first need to know that $u^\ep$ are uniformly bounded.  It is interesting that it seems in most-- if not all-- of the existing literature, this bound comes from basic observations using directly the structure of the equation being homogenized.  For (\ref{eqIntro:MainEpScale}), however, this is not the case, and such an estimate will not be true without the centering condition on $B$, Assumption \ref{assume:ABCompatible}.  Thus, the estimate we need for the Neumann problem in fact utilizes the homogenization of (\ref{eqIntro:MainEpScale}) with \emph{regular Dirichlet data}. 

\begin{lem}[$L^\infty$ bounds]\label{lem:uEpLInfinity}
	There exists a universal $C>0$ such that $\displaystyle\norm{u^\ep}_{L^\infty(\Sigma^1)}\leq C\norm{g}_{L^\infty}(\Sigma_0)$, where $u^\ep$ is the unique solution of (\ref{eqIntro:MainEpScale}).
\end{lem}

\begin{proof}
	This statement is immediate from Proposition \ref{prop:NemannEstimates}.  If we define $\eta^\ep$ as
	\begin{align*}
		\eta^{\ep}(y):= \frac{1}{\ep}u^\ep(\ep y),
	\end{align*}
	then we see that $\eta^\ep$ satisfies (\ref{eqBarr:EqForEtaEp}).  Indeed, 
	\begin{align*}
		\abs{\partial_n \eta^\ep(y)} = \abs{\partial_n u^\ep(\ep y)} = \abs{g(y)}\leq \norm{g}_{L^\infty(\Sigma_0)}\ \text{for}\ y\in\Sigma_0,
	\end{align*}
	and
	\begin{align*}
		\abs{\eta^\ep(y)} = \abs{u^\ep(y)} = 0\ \text{for}\ y\in\Sigma_{1/\ep}.
	\end{align*}
	Also, the rescaling of $u^\ep$ in this way transforms the equation in (\ref{eqIntro:MainEpScale}) in $\Sigma^1$ into (\ref{eqBarr:EqForEtaEp}) in $\Sigma^{1/\ep}$.  Hence Proposition \ref{prop:NemannEstimates} implies that
	\begin{align*}
		\norm{u^\ep}_{L^\infty(\Sigma^1)}\leq \ep\norm{\eta^\ep}_{L^\infty(\Sigma^{1/\ep})}\leq C\norm{g}_{L^\infty(\Sigma_0)}.
	\end{align*}
\end{proof}

By Proposition \ref{prop:C1GamRegNeumannProblem}, we have the following corollary.

\begin{cor}\label{cor:uEpHolderAndGradientHolder}
	$u^\ep$ satisfies the following estimates
	\begin{align*}
		[u^\ep]_{C^\gam(\Sigma^1)}\leq C\norm{g}_{L^\infty(\Sigma_0)}
	\end{align*}
	and
	\begin{align*}
		[u^\ep]_{C^{1,\gam}(\Sigma^1)}\leq \frac{C}{\ep^{\gam}}(1+ [g]_{C^\gam(\Sigma_0)}).
	\end{align*}
\end{cor}

\section{Structural Features of $\I^{1/\ep}$}\label{sec:StructureOfIep}

In this section, we collect several technical features of the operator, $\I^{1/\ep}$ (defined above, in (\ref{eqMainIdea:BulkDirichletAtMicroScale}), (\ref{eqMainIdea:IinMicroVars})), which will be used to prove the existence and uniqueness of the limiting constant, $\bar g$, in Section \ref{sec:HomogMainArgs}.  The reader can skip this before reading Section \ref{sec:HomogMainArgs}.  The most important result for this section is the following

\begin{lem}[Addition of constants]\label{lem:StructureAddConstants}
	There exists a function $f^\ep\in C(\Sigma_0)\intersect L^{\infty}(\Sigma_0)$ such that for all $\phi\in C^{1,\gam}(\Sigma_0)$ and constants, $c$, 
	\begin{align*}
		\I^{1/\ep}(\phi + c, y) = \I^{1/\ep}(\phi,y) - cf^\ep(y).
	\end{align*}
Furthermore, $f^\ep\geq 0$, and  for $0<c_1<c_2$ (the universal constants from Proposition \ref{prop:BarrierNonTrivNormalDeriv})
	\begin{align*}
		\ep c_1\leq f^\ep(y) \leq \ep c_2.
	\end{align*}
	
\end{lem}

\begin{rem}
	This is the counterpart to \cite[Lemma 3.6]{GuSc-2014NeumannHomogPart1DCDS-A}.  The difference is that there the operator in the bulk, $F$, and subsequently $\I^{1/\ep}$ were invariant by addition of hyperplanes.  In that case, this lemma is trivial, as the equation is invariant by the addition of hyperplanes to solutions.  Here hyperplanes are not solutions, nor can they serve as sub/super solution barriers, and so we appeal to the barrier provided by Proposition \ref{prop:BarrierNonTrivNormalDeriv}.
\end{rem}

\begin{proof}[Proof of Lemma \ref{lem:StructureAddConstants}]
  Let us call $V^\ep_\phi$ the solution of (\ref{eqMainIdea:BulkDirichletAtMicroScale}) with boundary data, $\phi$, and $\tilde V^\ep_\phi$ the solution of (\ref{eqMainIdea:BulkDirichletAtMicroScale}) with data $\phi+c$.  If $\phi^\ep$ is as in Proposition \ref{prop:BarrierNonTrivNormalDeriv}, then since (\ref{eqMainIdea:BulkDirichletAtMicroScale}) is linear, we see that 
\begin{align*}
	\tilde V^\ep_\phi = V^\ep_\phi + \ep c\phi^\ep.  
\end{align*}
Hence,
\begin{align*}
	\I^{1/\ep}(\phi+c,y) = \partial_n\tilde V^\ep_\phi(y) = \partial_n V^\ep_\phi(y) + \ep c \partial_n\phi^\ep.
\end{align*}	
The lemma follows immediately from Proposition \ref{prop:BarrierNonTrivNormalDeriv} with the choice
\begin{align}\label{eqStructIep:fepDef}
	f^\ep = -\ep\partial_n\phi^\ep.
\end{align}
\end{proof}

\begin{lem}[Rescaling]\label{lem:Rescaling}
	For all $v\in C^{1,\gam}(\Sigma_0)$,
	\begin{align*}
		\I^{1}(\ep v(\frac{\cdot}{\ep}),x) = \I^{1/\ep}(v,\frac{x}{\ep})
	\end{align*}
\end{lem}

\begin{proof}
	We let $U^\ep_v$ solve (\ref{eqMainIdea:BulkDirichletAtMacroScale}) with data $\ep v(\cdot/\ep)$, let $V^\ep_v$ solve (\ref{eqMainIdea:BulkDirichletAtMicroScale}) with data $v$, and define 
	\begin{align*}
		\Tilde U^\ep(y) = \ep V^\ep_v(\frac{x}{\ep}).
	\end{align*}
	Hence 
	\begin{align*}
		\grad \Tilde U^\ep(x) = \grad V^\ep_v(\frac{x}{\ep}),\ \text{and}\ D^2 \Tilde U^\ep (x) = \frac{1}{\ep}D^2V^\ep_v(\frac{x}{\ep}),
	\end{align*}
	and thus $\Tilde U^\ep$ solves (\ref{eqMainIdea:BulkDirichletAtMacroScale}), with data $\ep v(\cdot/\ep)$.  By the uniqueness of solutions to (\ref{eqMainIdea:BulkDirichletAtMacroScale}), we conclude 
	\begin{align*}
		\Tilde U^\ep = U^\ep_v,
	\end{align*}
	and hence
	\begin{align*}
		\I^1(\ep v(\frac{\cdot}{\ep}),x) = \partial_n U^\ep_v(x) = \partial_n \Tilde U^\ep(x) = \partial_n V^\ep_v(\frac{x}{\ep}) = \I^{1/\ep}(v,x).
	\end{align*}
\end{proof}

    The following auxiliary functions will be useful for localizing points of maxima and minima.  Let
    \begin{equation*}
         \phi_1(x):=\frac{|x|^2}{1+|x|^2},\ \text{for}\ x\in\Sigma_0
    \end{equation*}
    and for $R>0$ we will consider the functions
    \begin{equation}\label{eqSetUp:Phi_Psi_def}
         \phi_R(x) := \phi(x/R).
    \end{equation}

    As we shall see below, the Dirichlet to Neumann maps for the standard extremal operators $\mathcal{M}^{\pm}$ will be of use (defined in the notation of Section \ref{sec:Notation}).  These D-to-N maps for the extremal operators are the ones that give the definition/representation of the appropriate extremal operators for the integro-differential setting of $\I^1$ and $\I^{1/\ep}$.  These integro-differential extremal operators are not essential for this paper, as we don't solve a nonlinear equation.  However, we choose to use them here for possible application in the future. They are defined as follows, given $\phi:\Sigma_0 \to \mathbb{R}$,  define,
   \begin{equation}\label{eqStructIep:NonlocalExtremalDef}
    	 M^{r,\pm}(\phi,y) : = \partial_n U_\phi^{r,\pm},
    \end{equation}
    where $U^{r,\pm}_\phi = U^{r,\pm}: \Sigma^r \to \mathbb{R}$ are the unique bounded viscosity solutions (see Proposition \ref{prop:UniqueViscSolInSigmaR}) of
	\begin{equation}\label{eqStructIep:ExtremalDirichletrScalePlus}
		\begin{cases}
			\mathcal{M}^{+}(D^2U^{r,+}) + C\abs{\grad U^{r,+}}  = 0 &\text{in}\ \Sigma^{r},\\
			U^{r,+}  = 0 &\text{on}\ \Sigma_{r},\\	
			U^{r,+}  = \phi &\text{on}\ \Sigma_0,	
		\end{cases}
	\end{equation}
and
	\begin{equation}\label{eqStructIep:ExtremalDirichletrScaleMinus}
		\begin{cases}
			\mathcal{M}^{-}(D^2U^{r,-}) - C\abs{\grad U^{r,-}}  = 0 \ &\text{in}\ \Sigma^{r},\\
			U^{r,-}  = 0\ &\text{on}\ \Sigma_{r},\\	
			U^{r,-} = \phi\ &\text{on}\ \Sigma_0.		
		\end{cases}
	\end{equation}
The operators $\M^{\pm}$ are the standard Pucci extremal operators of the second order fully nonlinear theory \cite[Chapter 2]{CaCa-95} (see also the notation in Section \ref{sec:Notation}), and they correspond to the linear ellipticity assumption that $\lam \Id\leq A(x)\leq \Lam\Id$.  We mention that by, e.g. \cite{SilvestreSirakov-2013boundary}, if $\phi\in C^{1,\gam}(\Sigma_0)$, then $U^{r,\pm}\in C^{1,\gam}(\overline{\Sigma^r})$.  Hence the definition (\ref{eqStructIep:NonlocalExtremalDef}) holds classically in the pointwise sense.

    \begin{prop}[Bump function]\label{prop:AuxiliaryFunctionsAreNice}
        
		Assume that $r\geq 1$ and $R_0>0$ are both fixed.  Then
		\begin{align*}
		\lim_{R\to\infty} \sup_{x\in B'_{R_0}} |M^{r,\pm}(\phi_R,x)| = 0	
		\end{align*}
		(we recall that $B'_{R_0}\subset\Sigma_0$).
    \end{prop}

	\begin{proof}[Proof of Proposition \ref{prop:AuxiliaryFunctionsAreNice}]
		We just focus on the case of $M^{r,+}(\phi)$. The proof of this proposition is a result of the stability of viscosity solutions of (\ref{eqStructIep:ExtremalDirichletrScalePlus}) with respect to local uniform limits of the Dirichlet data, combined with the $C^{1,\gam}$ boundary regularity in \cite{SilvestreSirakov-2013boundary}.  To this end, we let $U_{\phi_R}$ denote the solution of (\ref{eqStructIep:ExtremalDirichletrScalePlus}) with data $\phi_R$.  We note that for each $t$, fixed, 
		\begin{align*}
			\lim_{R\to\infty}\norm{\phi_R}_{L^\infty(B'_t)} = 0,
		\end{align*}
and since $\hat U (x)=0$ is the unique solution to (\ref{eqStructIep:ExtremalDirichletrScalePlus}) with $\hat U|_{\Sigma_0}=0$, the stability of (\ref{eqStructIep:ExtremalDirichletrScalePlus}) implies that $U_{\phi_R}\to 0$ locally uniformly in $\Sigma^r$.  Furthermore, since $\phi_R\in C^{1,\gam}(\Sigma_0)$, \cite{SilvestreSirakov-2013boundary} shows that $U_{\phi_R}\in C^{1,\gam}(\overline{\Sigma^r})$.  Thus, \cite[Theorem 1.1]{SilvestreSirakov-2013boundary} implies
\begin{align}\label{eqStructIep:NormalDerivEstimatesPhiR}
	\norm{\partial_n U_{\phi_R}}_{L^{\infty}(B'_{R_0})}\leq C(\norm{U_{\phi_R}}_{L^\infty(B^+_{2R_0})} + \norm{\phi_R}_{L^{\infty}(B'_{2R_0})} + \norm{\grad\phi_R}_{C^{\gam}(B'_{2R_0})}).
\end{align}
We note here that $B'_{R_0}\subset\Sigma_0$ and $B^+_{R_0}\subset \overline{\Sigma^r}$ are a ball and a half ball in respectively the boundary and the closure of $\Sigma^r$.  The first term on the right of (\ref{eqStructIep:NormalDerivEstimatesPhiR}) converges to $0$ by the stability of (\ref{eqStructIep:ExtremalDirichletrScalePlus}).  The second and third terms converge to $0$ by the rescaling that defines $\phi_R$.  Thus we conclude
		\begin{align*}
		\lim_{R\to\infty} \sup_{x\in B'_{R_0}} |M^{r,+}(\phi_R,x)| = 0.	
		\end{align*}
	\end{proof}

\begin{lem}[Comparison principle for smooth functions]\label{lem:ClassicalComparisonForIOperator}
 Let $u,v : \Sigma_0 \to \mathbb{R}$ be bounded functions such that $\I^{1/\ep}(u,\cdot)$ and $\I^{1/\ep}(v,\cdot)$ are classically defined and
	\begin{equation*}
		\I^{1/\ep}(u,y) \geq \I^{1/\ep}(v,y) \;\;\forall\;y \in \Sigma_0.
	\end{equation*}
    Then,
    \begin{equation*}
         u(y) \leq v(y)\;\;\forall\;y\in\Sigma_0.
    \end{equation*}
\end{lem}

\begin{proof}[Proof of Lemma \ref{lem:ClassicalComparisonForIOperator}]
	Using Proposition \ref{prop:AuxiliaryFunctionsAreNice}, the proof of Lemma \ref{lem:ClassicalComparisonForIOperator} now follows identically to the one in \cite[Lemma 3.9]{GuSc-2014NeumannHomogPart1DCDS-A}.
	
\end{proof}

\begin{lem}[Dependence on right hand side]\label{lem:GlobalEpLevelBounds} 
	If $w\in C^{1,\gam}(\Sigma_0)$ solves
	\[
	\I^{1/\ep}(w,y) = g(y)\ \ \text{in}\ \Sigma_0,
	\]
	then
	\begin{equation*}
	-\frac{1}{c_1\ep}\norm{g}_{L^\infty} \leq w \leq \frac{1}{c_1\ep}\norm{g}_{L^\infty},
	\end{equation*}
	where $c_1$ is the constant from Lemma \ref{lem:StructureAddConstants}.
\end{lem}

\begin{proof}[Proof of Lemma \ref{lem:GlobalEpLevelBounds}]
We note that since $\I^{1/\ep}(0,\cdot)=0$, it follows by Lemma \ref{lem:StructureAddConstants} that for any constant, $c$,
\begin{align*}
	\I^{1/\ep}(c,y) = -cf^\ep(y).
\end{align*}
Using $c_1$ from Lemma \ref{lem:StructureAddConstants}, we have that
\begin{align*}
	\I^{1/\ep}(w,y) = g(y)
	\leq \norm{g}_{L^\infty}
	=\frac{\ep c_1\norm{g}}{\ep c_1} 
	\leq f^\ep(y)\frac{\norm{g}}{c_1\ep}
	=\I^{1/\ep}\left(\frac{-1}{c_1\ep}\norm{g}_{L^\infty},y\right).
\end{align*}
Thus by the comparison for $\I^{1/\ep}$ (Lemma \ref{lem:ClassicalComparisonForIOperator}), we conclude that
\begin{align*}
	\frac{-\norm{g}}{c_1\ep} \leq w.
\end{align*}
The reverse inequality follows analogously.
\end{proof}

The following monotonicity property with respect to the domain, $\Sigma^{1/\ep}$, will be useful in Section \ref{sec:HomogMainArgs}. 
\begin{lem}[Monotonicity with respect to domain]\label{lem:IrComparison}
	Suppose that $(1/\ep_2)\geq (1/\ep_1)$ and that $u\geq 0$, then
	\begin{equation*}
	     \I^{1/\ep_2}(u,y)\geq \I^{1/\ep_1}(u,y)\ \ \forall\ y\in\Sigma_0.
    \end{equation*}
\end{lem}

\begin{proof}
	Let $V^{\ep_1}_u$ and $V^{\ep_2}_u$ solve (\ref{eqMainIdea:BulkDirichletAtMicroScale}) in respectively $\Sigma^{1/\ep_1}$ and $\Sigma^{1/\ep_2}$.
	Note that $\Sigma^{1/\ep_1} \subset \Sigma^{1/\ep_2}$, so $V^{\ep_2}_u$ is defined in $\Sigma^{1/\ep_1}$. Since $V^{\ep_2}_u = u \geq 0$ on $\Sigma_0$, the comparison principle implies that $V^{\ep_2}_u\geq 0$ in $\Sigma^{1/\ep_2}$, and in particular $V^{1/\ep_2}_u \geq 0$ on $\Sigma_{1/\ep_1}$. Moreover, $V^{\ep_1}_u$ and $V^{\ep_2}_u$ agree on $\Sigma_0$ and solve the same equation in $\Sigma^{1/\ep_1}$. Thus $V^{\ep_2}_u$ is a supersolution for the equation solved by $V^{\ep_1}_u$, so that $V^{\ep_1}_u\leq V^{\ep_2}_u$ everywhere in $\Sigma^{1/\ep_1}$.
	
	Since the two functions agree on $\Sigma_0$, their normal derivatives must be ordered, namely
	\begin{equation*}
		 \partial_n V^{\ep_2}_u(y) \geq \partial_n V^{\ep_1}_u(y)\ \ \forall\ y\in \Sigma_0,
	\end{equation*}
	and the lemma follows.
	
	\end{proof}

\begin{lem}\label{lem:RHSComparison}
 Let $(1/\ep)\geq 1$ be fixed and assume $f^\ep$ is as in Lemma \ref{lem:StructureAddConstants}. Suppose that there exist bounded $w_1,w_2\in C^{1,\gam}(\Sigma_0)$, respectively sub and super solutions $w_1$ and $w_2$ to
	\begin{align*}
		\I^{1/\ep}(w_1,y) \geq -b_1+ \frac{a_1f^\ep(y)}{\ep} + g(y)\ \ \text{and}\ \ \I^{1/\ep}(w_2,y) \leq  \frac{a_2f^\ep(y)}{\ep} + g(y)\ \  \text{in}\ \Sigma_0,
	\end{align*}
	where $b_1\geq0$.  Then $\displaystyle a_1-a_2 - (1/c_1)b_1\leq \ep\sup_{\Sigma_0}|w_1-w_2|$, where $c_1$ is the constant from Lemma \ref{lem:StructureAddConstants}.
\end{lem}

\begin{proof}
    Without loss of generality, assume that $a_1\geq a_2$, else there is nothing to prove (since $b_1\geq0$). Let $\tilde w_2 := w_2 -(1/\ep)(a_1-a_2)+(1/c_1\ep)b_1$, then by Lemma \ref{lem:StructureAddConstants} we have
    \begin{align*}
        \I^{1/\ep}(\tilde w_2,y) & = \I^{1/\ep}(w_2,y)+\frac{(a_1-a_2)}{\ep}f^\ep(y) - \frac{b_1f^\ep(y)}{c_1\ep} ,\\
							&\leq g(y)+\frac{a_2f^\ep(y)}{\ep} + \frac{(a_1-a_2)f^\ep}{\ep} - b_1\\
                           & \leq g(y) +\frac{a_1f^\ep(y)}{\ep} - b_1,\\
                           & \leq \I^{1/\ep}(w_1,y).
    \end{align*}
    Then Lemma \ref{lem:ClassicalComparisonForIOperator} yields that $w_1\leq \tilde w_2$, i.e.  $w_1 \leq \tilde w_2 = w_2-(1/\ep)(a_1-a_2)+(1/c_1\ep)b_1$. Rearranging,
    \begin{equation*}
          (1/\ep)(a_1-a_2) - (1/c_1\ep)b_1 \leq w_2-w_1 \leq \sup \limits_{\Sigma_0}|w_1-w_2|,
    \end{equation*}
    and the lemma follows by multiplying by $\ep$.
\end{proof}


\section{The Proofs of Theorems \ref{thm:LimitOfBoundaryValues} and  \ref{thm:Main}}\label{sec:HomogMainArgs}

We are now in a position to prove Theorem \ref{thm:LimitOfBoundaryValues} as well as show how it implies Theorem \ref{thm:Main}.  To this end, we define the new function
\begin{align*} 
	v^\ep=u^\ep|_{\Sigma_0}, 
\end{align*}	
where $u^\ep$ is the solution to (\ref{eqIntro:MainEpScale}).  We know by Corollary \ref{cor:uEpHolderAndGradientHolder} that $u^\ep\in C^{1,\gam}(\overline{\Sigma^1})$, and so  
\begin{align*}
	\partial_n u^\ep(x)=g(\frac{x}{\ep})\ \text{classically on}\ \Sigma_0.
\end{align*}
Furthermore, the uniqueness of solutions to (\ref{eqMainIdea:BulkDirichletAtMicroScale}) shows that when $V^\ep_{v^\ep}$ solves (\ref{eqMainIdea:BulkDirichletAtMicroScale}) with data $v^\ep$, then in fact 
\begin{align*}
	V^{\ep}_{v^\ep}\equiv u^\ep.
\end{align*}
Thus, as pointed out in Section \ref{sec:MainIdeas}, the original homogenization problem is equivalent to 
\begin{align*}
	\I^{1}(v^\ep, x) = g(\frac{x}{\ep})\ \ \text{in}\ \Sigma_0.
\end{align*}
It is useful to unscale this equation; thanks to Lemma \ref{lem:Rescaling} the function
\begin{align}\label{eqHomog:DefOfWep}
w^\ep(y)= (1/\ep)v^\ep(\ep y), 
\end{align}
solves
\begin{align}\label{eqHomog:IntDiffMicroVars}
	\I^{1/\ep}(w^\ep,y) = g(y)\ \ \text{in}\ \Sigma_0.
\end{align}
Since $\Sigma_0$ is irrational, $g|_{\Sigma_0}$ will not be periodic, but it will be almost periodic (Definition \ref{def:AlmostPeriodic}).  Furthermore, $\I^{1/\ep}$ is effectively an ``almost periodic operator'' on $\Sigma_0$, which is not precisely defined, but it manifests itself in the almost periodicity of $w^\ep$, presented below as Lemma \ref{lem:AlmostPeriodicW}.

\subsection{Almost periodicity of $w^\ep$}\label{sec:AlmostPeriodicI}

Here we use almost periodicity properties of $\Sigma_0$ with respect to $\integer^{d+1}$ to establish almost periodicity properties of $w^\ep$.  We remind the reader that the relevant definitions appear in Section \ref{sec:Assume}.

\begin{lem}\label{lem:AlmostPeriodicW}
	There exists a universal $C>0$ such that for all $\del>0$, there exists $R_\del>0$, such that if $E_\del$ is a $\del$-almost period of $g$ and $F_\del$ is a $\displaystyle C\left(\del + \frac{\del}{\ep}\right)$- almost period of $w^\ep$, then 
	\[
	B'_{R_\del}(z)\intersect \left(E_\del\intersect F_\del\right)\not=\emptyset\ \text{for all}\ z\in\Sigma_0.
	\]
	We emphasize that $R_\del$ depends only on the irrationality of $n$, and in particular is independent of $\ep$.
	Note that in this context, $B'_{R_\del}\subset\Sigma_0$.
\end{lem}

\begin{proof}
	Let $z\in\Sigma_0$ be fixed and $\rho>0$ be arbitrary.  We will eventually choose $\rho$ to depend on $\del$ to make the calculation easier.  First, we note that by the irrationality of $n$, there exists an $R_\rho$ such that given any $z\in\Sigma_0$, there is $\tau(z)\in\Sigma_0\intersect B_{R_\rho}(z)$ such that $dist(\tau(z),\integer^{d+1})\leq \rho$ (see, e.g. \cite[Lem. 2.7]{ChKiLe-2012HomogNeumann}).  We take $\hat \tau\in\Sigma^{1/\ep}\intersect\integer^{d+1}$ to be any element that is within $dist(\tau(z),\integer^{d+1})$ to $\Sigma_0$.  We also note that by possibly re-adjusting the choice of $\tau(z)$, we can assume that $\tau(z)\perp(\hat\tau-\tau(z))$.  Both $\tau(z)$ and $\hat \tau(z)$ depend on $z$, but we suppress this dependence for the rest of the proof.
	
	We begin by unscaling the original equation, (\ref{eqIntro:MainEpScale}).  To that end, define
	\begin{align*}
		V^\ep(y) = \frac{1}{\ep}u^\ep(\ep y).
	\end{align*}
Thus, $V^\ep$ solves the equation in the microscopic variables, i.e.
	\begin{align*}
		\begin{cases}
			\Tr(A(y)D^2V^\ep) + B(y)\grad V^\ep = 0\ &\text{in}\ \Sigma^{1/\ep}\\
			V^\ep = 0\ &\text{on}\ \Sigma_{1/\ep}\\
			\partial_n V^\ep = g(y)\ &\text{on}\ \Sigma_0.
		\end{cases}
	\end{align*}
This means that in fact, $V^\ep|_{\Sigma_0}=w^\ep$, where $w^\ep$ is defined in (\ref{eqHomog:DefOfWep}) and solves (\ref{eqHomog:IntDiffMicroVars}).
	We will now shift $V^\ep$ by $\hat\tau$, defining  
	\begin{align*}
		\hat V^\ep(y) = V^\ep(y+\hat \tau).
	\end{align*}
We define
		\begin{align*}
			\hat z = \hat \tau - \tau,\ \text{note, by definition}\ \abs{\hat z}\leq \rho,
		\end{align*}
and because of the orthogonality of $\tau\perp (\hat\tau-\tau)$, we see that $\hat V^\ep$ solves in the shifted domain, $(\Sigma^{1/\ep}-\hat z)$,
	\begin{align*}
		\begin{cases}
		\Tr(A(y+\hat\tau)D^2\hat V^\ep) + B(y+\hat\tau)\grad \hat V^\ep=0\ &\text{in}\ (\Sigma^{1/\ep}-\hat z)\\
			\hat V^\ep = 0\ &\text{on}\ (\Sigma_{1/\ep}-\hat z)\\
			\partial_n \hat V^\ep = g(y+\hat \tau)\ &\text{on}\ (\Sigma_0 - \hat z).
		\end{cases}
	\end{align*}
Since by choice, $\hat\tau\in\integer^{d+1}$, and by the periodicity of $A$ and $g$ we conclude
	\begin{align*}
		\begin{cases}
			\Tr(A(y)D^2\hat V^\ep) + B(y)\grad \hat V^\ep=0\ &\text{in}\ (\Sigma^{1/\ep}-\hat z)\\
			\hat V^\ep = 0\ &\text{on}\ (\Sigma_{1/\ep}-\hat z)\\
			\partial_n \hat V^\ep = g(y)\ &\text{on}\ (\Sigma_0 - \hat z).
		\end{cases}
	\end{align*}

Now we can compare $V^\ep$ to $\hat V^\ep$.  Let us define
\begin{align*}
	\hat W^\ep = V^\ep - \hat V^\ep\ \text{in}\ \Sigma^{1/\ep}\intersect (\Sigma^{1/\ep}-\hat z) = \Sigma^{1/\ep-\abs{\hat z}}.
\end{align*}
We will need to observe that the oscillation of $V^\ep$ and $\hat V^\ep$ \emph{in the sets that are a distance less than one to $\Sigma_{1/\ep}$} are bounded independently of $\ep$.  This is true because we can use the barriers involving $\phi^\ep$ from Proposition \ref{prop:BarrierNonTrivNormalDeriv} plus the observation that $\phi^\ep$ stays a uniform distance from the fixed hyperplane, $1/\ep - y_{d+1}$.
Thus, from the $C^\gam$ estimates for the Dirichlet problem (near $\Sigma_{1/\ep}$) we know that $V^\ep$ and $\hat V^\ep$ are $C^\gam$ in a neighborhood of $\Sigma_{1/\ep}$, uniformly in $\ep$.  First, we check the boundary data on the $1/\ep-\abs{\hat z}$ boundary:
\begin{align*}
	\norm{V^\ep-\hat V^\ep}_{L^\infty(\Sigma_{1/\ep-\abs{\hat z}})} 
	= \norm{V^\ep-0}_{L^\infty(\Sigma_{1/\ep-\abs{\hat z}})}
	\leq C\abs{\hat z}^\gam.
\end{align*}
Next we check the normal derivative on the $\Sigma_0$ hyperplane:
\begin{align*}
	\norm{\partial_n V^\ep-\partial_n\hat V^\ep}_{L^\infty(\Sigma_0)} 
	= \norm{g - \partial_n \hat V^\ep}_{L^\infty(\Sigma_0)}
\end{align*}
 since $V^\ep = g$ on $\Sigma_0$.  We would like to transfer this inequality to an evaluation on $\Sigma_0-\hat z$ to utilize the boundary values of $\partial_n \hat V^\ep$. 
By Corollary \ref{cor:uEpHolderAndGradientHolder}, we have
\begin{align*}
	[\partial_n\hat V^\ep]_{C^\gam} = [\partial_n u^\ep(\ep\cdot)]_{C^\gam}\leq \frac{C}{\ep^\gam}\ep^\gam=C,
\end{align*}
and so for $y\in\Sigma_0$, $y-\hat z\in\Sigma_0-\hat z$
\begin{align*}
	&\abs{\partial_n\hat V^\ep(y)-g(y)} \leq \abs{\partial_n \hat V^\ep(y)-\partial_n \hat V^\ep(y-\hat z)} + \abs{\partial_n\hat V^\ep(y-\hat z)-g(y)} \\
	&= \abs{\partial_n \hat V^\ep(y)-\partial_n \hat V^\ep(y-\hat z)} + \abs{g(y-\hat z)-g(y)}\\
	&\leq C\abs{\hat z}^\gam.
\end{align*}
We used the fact that $\partial_n \hat V^\ep=g$ on $\Sigma_0-\hat z$, and so by the the H\"older regularity of $g$ and $\partial_n \hat V^\ep$, the difference on $\Sigma_0$ is of order $\abs{\hat z}^\gam$. (Note, $\partial_n\hat V^\ep=g$ on $\Sigma_{-\abs{\hat z}}$, not on $\Sigma_0$.)  Importantly, this constant, $C$, in the two estimates above, is independent of $\ep$.

Therefore,  $\hat W^\ep$ solves
\begin{align*}
	\begin{cases}
		\Tr(A(y)D^2\hat W^\ep) + B(y)\cdot\grad \hat W^\ep = 0\ &\text{in}\ \Sigma^{1/\ep-\abs{\hat z}}\\
		\abs{\hat W^\ep}\leq C\abs{\hat z}^\gam\ &\text{on}\ \Sigma_{1/\ep-\abs{\hat z}}\\
		\abs{\partial_n \hat W^\ep}\leq C\abs{\hat z}^\gam\ &\text{on}\ \Sigma_0.
	\end{cases}
\end{align*}
Thus, by Proposition \ref{prop:NemannEstimates}, we conclude, in particular, that
\begin{align*}
	\norm{\hat W^\ep}_{L^\infty(\Sigma_0)}\leq C\frac{\abs{\hat z}^\gam}{\ep}.
\end{align*}
Finally, we check, 
\begin{align*}
	\norm{w^\ep(\cdot+\tau)-w^\ep}_{L^\infty(\Sigma_0)}&=\norm{V^\ep(\cdot+\tau)- V^\ep}_{L^\infty(\Sigma_0)}\\
	&\leq \norm{V^\ep(\cdot+\hat\tau) - V^\ep(\cdot+\tau)}_{L^\infty(\Sigma_0)} + \norm{V^\ep(\cdot+\hat\tau) - V^\ep}_{L^\infty(\Sigma_0)}\\
	&\leq C\abs{\tau-\hat\tau}^\gam + \norm{\hat W^\ep}_{L^\infty(\Sigma_0)}\\
	&\leq C\abs{\hat z}^\gam + C\frac{\abs{\hat z}^\gam}{\ep}.
\end{align*}
We also confirm
\begin{align*}
	\norm{g(\cdot+\tau)-g}_{L^\infty(\Sigma_0)}&=
	\norm{g(\cdot+\hat\tau)-g(\cdot-\tau)}_{L^\infty(\Sigma_0)} + \norm{g(\cdot+\hat\tau)-g}_{L^\infty(\Sigma_0)}\\
	&\leq C\abs{\hat\tau-\tau}^\gam + 0\\
	&=C\abs{\hat z}^\gam,
\end{align*}
since by choice, $\hat\tau\in\integer^{d+1}$, and the periodicity of $g$.  Now, to conclude the proposition, we choose $\rho=\del^{1/\gam}$, and since $\abs{\hat z}\leq \rho$, we have shown that indeed $\tau\in B'_{R_\del}(z)$ is an almost period for $w^\ep$ and $g$.
\end{proof}

\begin{rem}
	We note that the proof of Proposition \ref{prop:BarrierNonTrivNormalDeriv} utilized a similar technique to that in \cite[Section 3]{ChoiKim-2014HomogNeumannJMPA}, where there it was also very important to translate the whole equation by $\hat \tau$ and appeal to the periodicity to keep the coefficients unchanged.
\end{rem}

\subsection{Limits for $\ep w^\ep - \ep w^\ep(0)$}\label{sec:NonlocalIshii}

The key lemma that establishes Theorem \ref{thm:LimitOfBoundaryValues}, in combination with the almost periodicity of Lemma \ref{lem:AlmostPeriodicW},  is a nonlocal version of the almost periodic arguments which appeared for Hamilton-Jacobi equations in Ishii's work, \cite{Ishii-2000AlmostPeriodicHJHomog}.  There are however, many differences between the Hamilton-Jacobi setting and our nonlocal setting here, and we give a slightly different argument.

\begin{lem}[Nonlocal Elliptic Version of Ishii \cite{Ishii-2000AlmostPeriodicHJHomog}]\label{lem:IshiiEpWEpDecay}
	$w^\ep$ defined in (\ref{eqHomog:DefOfWep}), (\ref{eqHomog:IntDiffMicroVars}) satisfies the decay
	\begin{equation}\label{BoundaryEq:EpWEpDecay}
		\norm{\ep w^\ep - \ep w^\ep(0)}_{L^\infty(\Sigma_0)} \to 0\ \ \text{as}\ \ \ep\to0.
	\end{equation}
\end{lem}

\begin{proof}[Proof of Lemma \ref{lem:IshiiEpWEpDecay}]
	Let $\{\varepsilon_k\}_{k}$ be a sequence such that $\varepsilon_k \to 0^+$, and let $\{ y_k\}_k $ be a sequence in $\Sigma_0$ such that for each $k$,
	\begin{align*}
	  |\varepsilon_k w^{\varepsilon_k}(y_k)-\varepsilon_k w^{\varepsilon_k}(0)|\geq \tfrac{1}{2}\| \varepsilon_k w^{\varepsilon_k}-\varepsilon_k w^{\varepsilon_k}(0)\|_{L^\infty(\Sigma_0)}.
	\end{align*}
	Let $\delta>0$ be given.   Let $F_\del$ be the set of $C(\del+\del/\ep)$- periods of $w^\ep$. Thanks to Lemma \ref{lem:AlmostPeriodicW}, there is some $R_\delta>0$ such that
	\begin{align*}  
	  \left ( z+ B_{R_\delta}\right ) \cap ( F_\del) \ne \emptyset\;\;\forall\;z\in \Sigma_0.
	\end{align*}
	We note that it is essential that $R_\del$ be strictly sublinear in $\ep$ for later purposes.  In this case, $R_\del$ is in fact independent of $\ep$, which is even better.

	Taking $\varepsilon = \varepsilon_k$, $z=y_k$ above, it follows that for each $k$ there is some $\tau_k$ which is a $C(\ep_k\del+\del)$-almost period for $\varepsilon_k w^{\varepsilon_k}$ and such that
	\begin{align}\label{BoundaryEq:Yk-TkBounded}
	  y_k-\tau_k \in B_{R_\delta}.
	\end{align}
	In particular,
	\begin{align*}
	  |\varepsilon_k w^{\varepsilon_k}(y_k)-\varepsilon_k w^{\varepsilon_k}(0)| & \leq |\varepsilon_k w^{\varepsilon}(y_k)-\varepsilon_k w^{\varepsilon_k}(y_k-\tau_k)|+|\varepsilon_k w^{\varepsilon_k}(y_k-\tau_k)-\varepsilon_k w^{\varepsilon_k}(0)|.
	\end{align*}	
	Since $\tau_k$ is a $C(\ep_k\del+\del)$-almost period for $\varepsilon_k w^{\varepsilon_k}$ the first quantity on the right is at most $C(\ep_k\delta+\del)$,  and also from (\ref{BoundaryEq:Yk-TkBounded}),
	\begin{align*}
	  |\varepsilon_k w^{\varepsilon_k}(y_k)-\varepsilon_k w^{\varepsilon_k}(0)| \leq C(\ep_k\delta+\del) + \osc \limits_{B_{R_\delta}}\{\varepsilon_k w^{\varepsilon_k}\} ,\;\;\forall\; k>0.
	\end{align*}
	Next, we note that
	\begin{align*}
	  \osc \limits_{B_{R_\delta}}\{ \varepsilon_k w^{\varepsilon_k} \} = \osc \limits_{B_{\varepsilon_k R_\delta}} \{ v^{\varepsilon_k}\}.
	\end{align*}
	Corollary \ref{cor:uEpHolderAndGradientHolder} guarantees that the functions $v^{\varepsilon}$ are $C^{\bar \gamma}$-continuous in $B_1$, uniformly in $\varepsilon$. Therefore (for each fixed $\delta>0$),
	\begin{align*}
	  \lim \limits_{\varepsilon \to 0^+}\osc \limits_{B_{\varepsilon R_\delta}} \{ v^{\varepsilon}\} = 0.	
	\end{align*}	
	(Here, it would be enough that $\lim \ep R_\del=0$ if it happened that $R_\del$ depended on $\ep$.)
    Given that $\varepsilon_k \to 0$, for every large enough $k$ (this possibly depending on $\delta$) we have
	\begin{align*}
	  \tfrac{1}{2}\| \varepsilon_k w^{\varepsilon_k}-\varepsilon_k w^{\varepsilon_k}(0)\|_{L^\infty(\Sigma_0)} \leq |\varepsilon_k w^{\varepsilon_k}(y_k)-\varepsilon_k w^{\varepsilon_k}(0)| \leq 2C\delta.   	
	\end{align*}
	That is (as the sequence $\varepsilon_k\to 0^+$ was arbitrary)
	\begin{align*}
	  \limsup \limits_{\varepsilon \to 0^+}	\| \varepsilon w^{\varepsilon}-\varepsilon w^{\varepsilon}(0)\|_{L^\infty(\Sigma_0)}\leq 4C\delta,
    \end{align*}		
	letting $\del \to 0^+$, the lemma follows.
\end{proof}

\begin{lem}[Existence of constant limits]\label{lem:IshiiExistenceOfConst}
	Given any $\ep_j\to0$, there exists a subsequence, $\ep_j'$ such that $v^{\ep'_j}\to c$ uniformly on $\Sigma_0$, for some constant $c$.
\end{lem}

\begin{proof}[Proof of Lemma \ref{lem:IshiiExistenceOfConst}]
	By Proposition Corollary \ref{cor:uEpHolderAndGradientHolder} we know that $v^\ep\in C^{\gam}(\Sigma_0)$ for some $0< \gam < 1$. Thus since $v^\ep$ are uniformly bounded, we can take some subsequence such that $v^{\ep_j}(0)\to c$.
	Furthermore, Lemma \ref{lem:IshiiEpWEpDecay} shows that
	\begin{equation*}
		\norm{v^{\ep_j'} - v^{\ep_j'}(0)}_{L^\infty(\Sigma_0)} \to 0\ \text{as}\ \ep\to0.
	\end{equation*}  
	Hence $v^{\ep_j'}\to c$ uniformly on $\Sigma_0$.   
\end{proof}

\subsection{Uniqueness of the limiting constant}\label{sec:UniqueLimitConstant}
\begin{lem}\label{lem:UniquenessOfConst}
	The constant, $c$, of Lemma \ref{lem:IshiiExistenceOfConst} is independent of the sequence, $\ep_j$, and hence unique.
\end{lem}

\begin{rem}
	We note that this is basically a consequence of the fact that (\ref{eqHomog:IntDiffMicroVars}) is a uniformly elliptic integro-differential equation, and follows analogously to the arguments in e.g. Evans \cite{Evan-92PerHomog} or Ishii \cite{Ishii-2000AlmostPeriodicHJHomog}.  However, the proof is not as straightforward as in the existing literature for either second order or nonlocal elliptic equations due to the influence of the Dirichlet condition $u^\ep=0$ on $\Sigma_{1}$.  The necessary modifications are not serious road blocks, but we do include them for completeness. 
\end{rem}

\begin{proof}[Proof of Lemma \ref{lem:UniquenessOfConst}]
	Let $a_1$ and $a_2$ be constants such that there are sequences $v^{\ep_j}\to a_1$ and $v^{\ep_k}\to a_2$ uniformly on $\Sigma_0$.  We will establish that
	\[
	a_2\leq a_1,
	\]  
	and since the sequences were arbitrary, this proves the lemma.
	If we rewrite $v^{\ep_j}$ and $v^{\ep_k}$ in the microscale variables, this says that (recall $w^\ep$ from (\ref{eqHomog:DefOfWep}), (\ref{eqHomog:IntDiffMicroVars}))
	\begin{equation*}
		\ep_j w^{\ep_j} \to a_1\ \ \text{and}\ \ \ep_k w^{\ep_k}\to a_2\ \ \text{uniformly on}\ \Sigma_0.
	\end{equation*}
	We will also define the functions 
	\begin{equation*}
		 \hat w^{\ep_j}= w^{\ep_j} - \frac{1}{\ep_j}a_1\ \ \text{and}\ \ \hat w^{\ep_k} = w^{\ep_k} - \frac{1}{\ep_k}a_2.
	\end{equation*}
	In anticipation of applying Lemma \ref{lem:IrComparison}, we need to make sure that $\hat w^{\ep_j}$ and $\hat w^{\ep_k}$ are non-negative.  We do so by shifting them up by respectively $\del_j$, $\del_k$ where
	\begin{equation*}
		\del_j =\norm{\hat w^{\ep_j}}\ \ \text{and}\ \ \del_k=\norm{\hat w^{\ep_k}},
	\end{equation*}
	which gives
	\begin{equation*}
		\hat w^{\ep_j}+\del_j\geq0\ \text{and}\ \hat w^{\ep_k}+\del_k\geq0.
	\end{equation*}
	
	We will assume without loss of generality that $j$ and $k$ are such that $\ep_j<\ep_k$, which suffices because $j$ and $k$ can otherwise be chosen independently of one another.  Using the equations for $w^{\ep_j}$ and $w^{\ep_k}$, we see that from Lemmas \ref{lem:StructureAddConstants} and \ref{lem:IrComparison}
	\begin{align*}
		&\frac{(-\ep_k\del_k+ a_2)f^{\ep_j}(y)}{\ep_j} + g(y)=\\
		&=(-\ep_k\del_k + a_2)\left(\frac{f^{\ep_j}(y)}{\ep_j}-\frac{f^{\ep_k}(y)}{\ep_k}\right)+\frac{(-\ep_k\del_k + a_2)f^{\ep_k}(y)}{\ep_k} + g(y)\\
		&=\I^{1/\ep_k}(w^{\ep_k}-\frac{a_2}{\ep_k}+\del_k,y) + (-\ep_k\del_k + a_2)\left(\frac{f^{\ep_j}(y)}{\ep_j}-\frac{f^{\ep_k}(y)}{\ep_k}\right) \\
		 &= \I^{1/\ep_k}(\hat w^{\ep_k}+\del_k,y) +  (-\ep_k\del_k + a_2)\left(\frac{f^{\ep_j}(y)}{\ep_j}-\frac{f^{\ep_k}(y)}{\ep_k}\right)\\
		&\leq \I^{1/\ep_j}(\hat w^{\ep_k} + \del_k,y) +  (-\ep_k\del_k + a_2)\left(\frac{f^{\ep_j}(y)}{\ep_j}-\frac{f^{\ep_k}(y)}{\ep_k}\right)\\
		&\leq  \I^{1/\ep_j}(\hat w^{\ep_k} + \del_k,y) + \rho_{k,j}, 
	\end{align*}
	where
	\begin{align*}
		\rho_{j,k}=\max\left\{0,\sup_y(-\ep_k\del_k + a_2)\left(\frac{f^{\ep_j}(y)}{\ep_j}-\frac{f^{\ep_k}(y)}{\ep_k}\right)\right\}.
	\end{align*}
	But on the other hand,
	\begin{align}
		\frac{(-\ep_j\del_j+ a_1)f^{\ep_j}(y)}{\ep_j} + g(y) &= \I^{1/\ep_j}(\hat w^{\ep_j}+\del_j,y).
	\end{align}
	Thus
	\begin{align*}
		&\I^{1/\ep_j}(\hat w^{\ep_j}+\del_j,y)\leq\frac{(-\ep_j\del_j+ a_1)f^{\ep_j}(y)}{\ep_j} + g(y)\\
		\intertext{and}
		&\I^{1/\ep_j}(\hat w^{\ep_k} + \del_k,y)\geq \frac{(-\ep_k\del_k+ a_2)f^{\ep_j}(y)}{\ep_j} + g(y) - \rho_{j,k},
	\end{align*}
	where $\rho_{j,k}\geq0$ is defined above.

	Thus Lemma \ref{lem:RHSComparison} tells us that
	\begin{align*}
		(-\ep_k\del_k + a_2) - (-\ep_j\del_j + a_1) - (1/c_1)\rho_{j,k} &\leq \ep_j \sup_{\Sigma_0}((\hat w^{\ep_k}+\del_k) - (\hat w^{\ep_j}+\del_j)).
	\end{align*}
	Hence
	\begin{align}
		a_2 - a_1 &\leq \ep_j\sup_{\Sigma_0} \hat w^{\ep_k} - \ep_j\inf_{\Sigma_0} \hat w^{\ep_j} +\ep_j\del_k - \ep_j\del_j + \ep_k\del_k-\ep_j\del_j + (1/c_1)\rho_{j,k} \nonumber\\
		 &\leq \ep_j \norm{\hat w^{\ep_k}}_{L^\infty(\Sigma_0)} + \ep_j\norm{\hat w^{\ep_j}}_{L^\infty(\Sigma_0)} + \ep_j\del_k + \ep_k\del_k +(1/c_1)\rho_{j,k}  \nonumber\\
		 &\leq 3\ep_k \norm{\hat w^{\ep_k}}_{L^\infty(\Sigma_0)} + \ep_j\norm{\hat w^{\ep_j}}_{L^\infty(\Sigma_0)} + (1/c_1)\rho_{j,k}, \label{eqBoundary:Lem4.5e1}
	\end{align}
	where we have used both $\ep_j\del_j\geq0$ and $\ep_j<\ep_k$.

	Now, we recall the definition of $f^\ep$, from the proof of Lemma \ref{lem:StructureAddConstants}, in (\ref{eqStructIep:fepDef}) as
	\begin{align*}
		f^\ep=-\ep\partial_n\phi^\ep.
	\end{align*} 
	Thanks to Lemma \ref{lem:LimitInInfiniteDomain}, we know that $\phi^\ep-1/\ep\to\psi^{\infty}$ uniformly on $\overline{\Sigma^1}$, and since $\phi^{1/\ep}-1/\ep$ also has uniform $C^{2,\gam}(\overline{\Sigma^1})$ estimates, it holds that 
	\begin{align*}
		\frac{f^\ep}{\ep}\to \partial_n \psi^\infty\ \text{uniformly on}\ \Sigma_0.
	\end{align*}
	Hence 
	\begin{align*}
		\rho_{j,k}\to0\ \ \text{as}\ j\to\infty\ \text{and}\ k\to\infty.
	\end{align*}
	
	  Now, we preserve $\ep_j<\ep_k$ and first take $j\to\infty$ followed by $k\to\infty$.  By construction of $\hat w^{\ep_j}$ and $\hat w^{\ep_k}$, we have
	\[
	\ep_k \norm{\hat w^{\ep_k}}_{L^\infty(\Sigma_0)}\to 0\ \ \text{and}\ \ \ep_j \norm{\hat w^{\ep_j}}_{L^\infty(\Sigma_0)}\to0.
	\]
	Hence $a_2\leq a_1$, and this finishes the lemma.
\end{proof}

\subsection{The Proofs of Theorems \ref{thm:LimitOfBoundaryValues} and \ref{thm:Main}}

Theorem \ref{thm:LimitOfBoundaryValues} follows immediately from Lemmas \ref{lem:IshiiExistenceOfConst} and \ref{lem:UniquenessOfConst}.  We again mention that by Corollary \ref{cor:uEpHolderAndGradientHolder}, $\norm{u^\ep}_{C^{\gam}}\leq C$ independently of $\ep$.  Thus, we can extract locally uniformly convergent subsequences of $u^\ep$ in $\Sigma^1$.  Let $\bar u$ be \emph{any} possible subsequential limit of $u^\ep$. The perturbed test function method, as in Proposition \ref{prop:PTFM}, shows that $\bar u$ is a solution of 
\begin{align*}
	\Tr(\bar A D^2\bar u) = 0\ \text{in}\ \Sigma^1,
\end{align*}
(see Section \ref{sec:Perturbed}).
Theorem \ref{thm:LimitOfBoundaryValues} gives a unique constant, $\bar c$ such that $\bar u|_{\Sigma_0}=\bar c$.  The uniform H\"older continuity of $u^\ep$ gives $\bar u|_{\Sigma_1}=0$.  Thus by the uniqueness of solutions to (\ref{eqIntro:MainEffective}) we see that there is exactly one choice for $\bar u$.  Hence Theorem \ref{thm:Main} is established with $\bar g=-\bar c$.


\section{Modifications to Obtain The Cell Problem of Choi-Kim \cite{ChoiKim-2014HomogNeumannJMPA}}\label{sec:ChoiKim}

The work of Choi and Kim \cite{ChoiKim-2014HomogNeumannJMPA} proves the homogenization of fully nonlinear equations with oscillatory Neumann data in some more general domains.  They studied (\ref{eqIntro:MainEpScale}) with a nonlinear operator in the interior, as
\begin{align}
	\begin{cases}
		F(\frac{x}{\ep},D^2u^\ep)=0\ &\text{in}\ \Om\\
		u^\ep = 0\ &\text{on}\ K\\
		\partial_n u^\ep(x)=g(\frac{x}{\ep})\ &\text{on}\ \partial\Om.
	\end{cases}
\end{align} 
There are basically two main results that they establish.  The first is to obtain for each possible normal, $n$, the constant, $\bar g(n)$, from Theorem \ref{thm:Main} in, $\Sigma^1(n)$, where now it is written explicitly that $\bar g(n)$ depends on the normal direction.  This is basically a cell problem for the general domain.  The second is to study the continuity properties of $\bar g(n)$ with respect to $n$.  

With minor adaptations of our proofs above, we can also obtain this first result about the cell problem in \cite{ChoiKim-2014HomogNeumannJMPA}.  That is, we also have the following theorem (cf. \cite[Theorem 3.1]{ChoiKim-2014HomogNeumannJMPA})

\begin{thm}\label{thm:ChoiKimCellProblem}
	Assume that $F$ is $\integer^{d+1}$ periodic in $x$, uniformly elliptic, satisfies basic assumptions for existence / uniqueness of $u^\ep$, and $F(x,0)\equiv 0$.  Assume that instead of (\ref{eqIntro:MainEpScale}), $u^\ep$ now solves
	\begin{align}
		\begin{cases}
			F(\frac{x}{\ep},D^2u^\ep) = 0\ &\text{in}\ \Sigma^1\\
			u^\ep = 0\ &\text{on}\ \Sigma_1\\
			\partial_n u^\ep(x)=g(\frac{x}{\ep})\ &\text{on}\ \Sigma_0.
		\end{cases}
	\end{align}
	Then the outcome Theorem \ref{thm:Main} still holds true, where instead $\bar u$ solves $\Bar F(D^2\bar u)=0$ in $\Sigma^1$, and $\bar F$ is the same effective equation from the standard periodic homogenization theory (see \cite{Evan-92PerHomog}).
\end{thm}

There are only two statements / techniques in our proof of Theorem \ref{thm:Main} that need to be slightly modified to obtain Theorem \ref{thm:ChoiKimCellProblem}: the statement of Lemma \ref{lem:StructureAddConstants} and the proof of Lemma \ref{lem:AlmostPeriodicW}.  Due to the fact that $F(x,0)\equiv 0$, any affine function is a solution of $F(y,D^2u)=0$.  Thus, Proposition \ref{prop:BarrierNonTrivNormalDeriv} is in fact trivial in this case.  The function $\phi^\ep$, from Proposition \ref{prop:BarrierNonTrivNormalDeriv}, is simply
\begin{align*}
	\phi^\ep(y) = 1/\ep - y_{d+1}.
\end{align*}
Thus, for this case, Lemma \ref{lem:StructureAddConstants} is true with (see (\ref{eqStructIep:fepDef}))
\begin{align*}
	f^\ep(y) = \ep.
\end{align*}
Finally, in the proof of Lemma \ref{lem:AlmostPeriodicW}, we see that the function, $\hat W^\ep$, solves in the viscosity sense:
\begin{align*}
	\begin{cases}
		\M^-(D^2\hat W^\ep) \leq 0\ &\text{in}\ \Sigma^{1/\ep-\abs{\hat z}}\\
		\M^+(D^2\hat W^\ep) \geq 0\ &\text{in}\ \Sigma^{1/\ep-\abs{\hat z}}\\
		\abs{\hat W^\ep}\leq C\abs{\hat z}^\gam\ &\text{on}\ \Sigma_{1/\ep-\abs{\hat z}}\\
		\abs{\partial_n \hat W^\ep}\leq C\abs{\hat z}^\gam\ &\text{on}\ \Sigma_0.
	\end{cases}
\end{align*}
Here $\M^\pm$ are the Pucci extremal operators of fully nonlinear elliptic equations, see \cite[Chp. 2]{CaCa-95}.
We recall also that $\hat z$ was chosen so that $\abs{\hat z}^{\gam}\leq\del$.  Thus the functions 
\begin{align*}
	\eta^\ep_{super}(y) = C\del(1/\ep - y_{d+1}) + C\del\\ 
	\intertext{and} 
	\eta^\ep_{sub}(y) =  -C\del(1/\ep - y_{d+1}) - C\del
\end{align*}
are respectively upper and lower barriers for $\hat W^\ep$.  Thus it follows that 
\begin{align*}
	\abs{\hat W^\ep}\leq C\frac{\del}{\ep}.
\end{align*}
The rest of the proofs now follow in the same fashion as for the linear case proved in Section \ref{sec:HomogMainArgs}.  We note that nowhere else in Sections \ref{sec:StructureOfIep} and \ref{sec:HomogMainArgs} was linearity used.  In fact, the rest of the details are very similar to those in \cite{GuSc-2014NeumannHomogPart1DCDS-A}.


\section{Two Natural Questions}

In the proof of Lemma \ref{lem:AlmostPeriodicW}, it was very convenient that $A$ and $B$ were periodic so that the translation by $\hat \tau$ kept the equations for $V^\ep$ and $\hat V^\ep$ the same.  However, it seems that this should only be a convenience, and that in fact one could establish Lemma \ref{lem:AlmostPeriodicW} when $A$, $B$, and $g$ are only almost periodic functions.  Is this true?

When the non-divergence equation in (\ref{eqIntro:MainEpScale}) is replaced by the operator, $\Div(A(x/\ep)\grad u^\ep)=0$, the equation still has a comparison principle, even with bounded, measurable $A$ (see \cite[Chp. 8]{GiTr-98} or \cite{KiSt-1980BookIntroVarIneq}).  Thus the idea of using the D-to-N mapping is still plausible, but the main drawback could be the $C^{1,\gam}$ regularity (i.e. Proposition \ref{prop:C1GamRegNeumannProblem} and Corollary \ref{cor:uEpHolderAndGradientHolder}).  Assume that $A$ is only $C^\gam$; can the integro-differential method be utilized to cover the case of the oscillatory oblique Neumann condition?


\appendix

\section{Various Useful Facts and Extra Details of Auxiliary Results}

Here we collect some useful facts, expand upon the invariant measure, $m$, show a result about rates, and give the details of the perturbed test function method as it pertains to the \emph{interior} homogenization of (\ref{eqIntro:MainEpScale}).

\subsection{Useful Facts}
In this Appendix, we collect many facts, some of which are known, but we are unaware of particular references for their presentation.  Other results that appear in the appendix are ones that are known under slightly different assumptions than those given in Assumptions \ref{assume:A}--\ref{assume:ABCompatible}, and so we present them here for completeness.  The first is a theorem about uniqueness for various equations in possibly unbounded domains and it is rephrasing of the results of \cite[Sec. 7]{Ishii-1989ExistenceUniquenessCPAM} to the case of bounded solutions and uniformly elliptic equations (see also \cite[Sec. V.1]{IshiiLions-1990ViscositySolutions2ndOrder} for comments about strictly elliptic equations).

\begin{prop}\label{prop:UniqueViscSolInSigmaR}
	For all $r>0$, and fixed, equations (\ref{eqStructIep:ExtremalDirichletrScalePlus}) and (\ref{eqStructIep:ExtremalDirichletrScaleMinus}) admit a unique bounded viscosity solution.
\end{prop}

\begin{proof}[Sketch of Proposition \ref{prop:UniqueViscSolInSigmaR}]
	The main point of this result is to show a comparison principle between bounded viscosity sub and super solutions of e.g. (\ref{eqStructIep:ExtremalDirichletrScalePlus}).  The existence follows from Perron's method in Ishii \cite{Ishii-1987PerronMethodDUKE}.  The existence is a simplification of Ishii \cite[Thm. 7.1]{Ishii-1989ExistenceUniquenessCPAM} to the case of bounded solutions and uniformly elliptic equations.  First, we note that since (\ref{eqStructIep:ExtremalDirichletrScalePlus}) (and also (\ref{eqStructIep:ExtremalDirichletrScaleMinus})) are uniformly elliptic, we can always perturb sub/super solutions by a $C^{2,\al}$ function in $\Sigma^r$ so that the sub/super solution inequality is strict.  Second, the comparison between sub and super solutions when at least one of them is a \emph{strict} sub/super solution is an easier case, and follows from a simple modification of the proof of \cite[Thm. 7.1]{Ishii-1989ExistenceUniquenessCPAM}.  We can follow the proof exactly, and we note that in the stage of subtracting the function $g_R$ to force a max of $u^\ep-v_\ep-\psi(x,y)$ (keeping with the notation of \cite[Thm. 7.1]{Ishii-1989ExistenceUniquenessCPAM}), thanks to the boundedness of $u^\ep-v_\ep$, it suffices to subtract a small multiple, $\del g_R$, instead of $g_R$, and a maximum will still be achieved.  Then, after the application of \cite[Prop. 5.1]{Ishii-1989ExistenceUniquenessCPAM}, the desired contradiction is reached thanks to the assumption that the sub/super solutions were strict and from the fact that the gradient of $\del g_R$ can be made as small as necessary, for appropriately small $\del$.
\end{proof}

The following can be adapted from \cite{CrIsLi-92}, combined with some details in \cite[Appendix A]{GuSc-2014NeumannHomogPart1DCDS-A}
\begin{prop}[Comparison for the Neumann problem]\label{prop:NeumannComparison}
	Assume that $\M^\pm$ are the Pucci operators given in Section \ref{sec:Notation}; and that $u$ (``subsolution'') and $v$ (``supersolution'') solve the following in the viscosity sense:
	\begin{align*}
		\M^+(D^2u) + C\abs{\grad u}\geq F(D^2u, \grad u, x)\geq 0\ \text{in}\ \Sigma^1,\ \partial_n u=f\ \text{on}\ \Sigma_0,\ \text{and}\ u=u_0\ \text{on}\ \Sigma_1
	\end{align*}
	\begin{align*}
		\M^-(D^2v) - C\abs{\grad v}\leq F(D^2v, \grad v, x)\leq 0\ \text{in}\ \Sigma^1,\ \partial_n v=g\ \text{on}\ \Sigma_0,\ \text{and}\ v=v_0\ \text{on}\ \Sigma_1.
	\end{align*}
	If $f-g\geq 0$ on $\Sigma_0$ and $u_0\leq v_0$ on $\Sigma_1$, then $u\leq v$ in $\overline{\Sigma^1}$.  
\end{prop}

The next result is simply a paraphrasing of those that appear in \cite[Sec. 2,3,4]{LiebermanTrudinger-1986NonlinearObliqueBCTAMS}, and standard modifications of the arguments in \cite{LiebermanTrudinger-1986NonlinearObliqueBCTAMS} yield the result as stated here.  We note that for our purposes, there is no harm in using a H\"older exponent for $g$ that is lower than the one obtained in the Krylov-Safonov theorem.  Thus, in the result below, we may assume it is the same $\gam$ appearing for the H\"older norms for $g$, $v$, $\grad v$.
\begin{prop}[Lieberman-Trudinger \cite{LiebermanTrudinger-1986NonlinearObliqueBCTAMS}]\label{prop:C1GamRegNeumannProblem}
	Assume that $\lam \Id\leq A \leq\Lam\Id$, $A$, $B$, and $g$ are all bounded and $C^\gam$ continuous, and that $v$ solves
	\begin{align*}
	\begin{cases}
		\Tr(A(y)D^2 v) + B(y)\cdot\grad v =0\ &\text{in}\ \Sigma^1\\
		\partial_n v = g\ &\text{on}\ \Sigma_0.
	\end{cases}
	\end{align*}
	There exists a universal constant, $C_1(\lam,\Lam,d)$ such that
	\begin{align*}
		[v]_{C^{\gam}(\overline{\Sigma^{1/2}})} \leq C_1\left(\osc_{\Sigma^1}(v) + \norm{g}_{L^\infty(\Sigma_0)}\right)
	\end{align*}
	and
	\begin{align*}
		[v]_{C^{1,\gam}(\overline{\Sigma^{1/2}})}\leq C_1\left(
		\osc_{\Sigma^1}(v) + \norm{g}_{C^{\gam}(\Sigma_0)} 
		\right).
	\end{align*}
\end{prop}

The next theorem is essential for our method.  It is not clear that it is essential for the homogenization to occur, but we use it in a critical way (but maybe there is another approach without it).  As it appears in \cite{BeLiPa-78}, $A$ and $B$ should both be $C^{1,\gam}$, however, this is not necessary, and in Appendix \ref{sec:Rates} we give some details.

\begin{prop}[Rate of Homogenization {\cite[Chp. 3, Sec. 5, Thm 5.1]{BeLiPa-78}}]\label{prop:RatesHomogGlobal}
	There exists a universal constant, $C$, such that if $f\in C^{4,\gam}(\Sigma_0)$, $\bar A$ is the unique homogenized coefficients of (\ref{eqPerturb:ABarDef})  and Proposition \ref{prop:PTFM}, $w^\ep$ and $\bar w$ are the unique solutions of respectively 
	\begin{align*}
		\begin{cases}
		\Tr(A(\frac{x}{\ep})D^2w^\ep) + \frac{1}{\ep}B(\frac{x}{\ep})\cdot\grad w^\ep = 0\ &\text{in}\ \Sigma^1\\
		w^\ep = 0\ &\text{on}\ \Sigma_1\\
		w^\ep = f\ &\text{on}\ \Sigma_0
		\end{cases}
	\end{align*}
	and
	\begin{align*}
		\begin{cases}
		\Tr(\Bar A D^2 \bar w) = 0\ &\text{in}\ \Sigma^1\\
		\bar w = 0\ &\text{on}\ \Sigma_1\\
		\bar w = f\ &\text{on}\ \Sigma_0,
		\end{cases}
	\end{align*}
	then
	\begin{align*}
		\norm{w^\ep - \bar w}_{L^\infty(\Sigma^1)}\leq C\ep.
	\end{align*}
\end{prop}

\begin{rem}
	This theorem also holds for the case of divergence equations, due to \cite{AvellanedaLin-1987CompactnessHomogDivEqCPAM}.
\end{rem}

\subsection{The invariant measure}\label{sec:InvariantMeasure}
A crucial tool for homogenization is the existence of an invariant measure for $L$, which is mentioned in Remark \ref{rem:InvariantMeasureExistsInL2}.  The existence and uniqueness follows almost identically to that of \cite[Chp. 3, Thm. 3.4]{BeLiPa-78}, and we include here the main ideas required to modify their proof to suit our assumptions that $A$ and $B$ are $C^\gam$.  The difference between the two arguments is minor; in \cite{BeLiPa-78}, they converted the equation to a divergence form operator, whereas here we simply use the relevant estimates for the non-divergence setting.

We first note that a reworking of the results 9.11-9.14 in \cite{GiTr-98} show that in the setting of periodic functions, when $f$ is periodic and $u$ is a \emph{periodic} strong solution of
\begin{align*}
	Lu=f\ \text{in}\ \real^{d+1},\ \text{where}\ Lu=\Tr(A(y)D^2u)+B(y)\cdot\grad u,
\end{align*}
then
	\begin{align*}
		\norm{u}_{W^{2,2}([0,1]^{d+1})}\leq C(\norm{u}_{L^2([0,1]^{d+1})}+\norm{f}_{L^2([0,1]^{d+1})})
	\end{align*}
and also for $\sig>0$ large enough,
	\begin{align*}
		\norm{u}_{W^{2,2}([0,1]^{d+1})}\leq \norm{Lu-\sig u}_{L^2([0,1]^{d+1})}.
	\end{align*}
Thus, in particular, any periodic $W^{2,2}([0,1]^{d+1})$ solution of
\begin{align*}
	-\sig u + Lu = f
\end{align*}
is unique.  For existence, we can take any \emph{periodic} and H\"older continuous approximation of $f$, say $f^\del$, such that $f^\del\to f$ in $L^2([0,1]^{d+1})$.  Thus, there exists a unique, bounded, continuous, and periodic solution $u^\del$ to
\begin{align*}
	-\sig u^\del + Lu^\del = f^\del,
\end{align*}
and moreover, $u^\del\in C^{2,\gam}$, so the equation holds classically.  Thus, the above $W^{2,2}$ estimates are applicable, and they show that $\{u^\del\}$ is Cauchy in $W^{2,2}([0,1]^{d+1})$.  Hence the limit of $u^\del$ gives existence for the equation.

Up to this point, we have basically outlined the details that show the operator,
\begin{align*}
	L_\sig u = -\sig u + Lu,
\end{align*}
has the property that $L_\sig ^{-1}$ is well defined,
\begin{align*}
	L_\sig^{-1}: L^2([0,1]^{d+1})\to L^2([0,1]^{d+1}),
\end{align*}
with the estimate
\begin{align*}
	\norm{L_\sig^{-1}f}_{W^{2,2}([0,1]^{d+1})}\leq C\norm{f}_{L^2([0,1]^{d+1})},
\end{align*}
and hence is compact.  The final remaining step in the Fredholm Alternative to obtain the existence of the invariant measure, $m$, is to determine the dimension of the space of solutions to 
\begin{align*}
	(I+\sig(L_\sig^{-1}))z=0,\ \text{for}\ z\in L^2([0,1]^{d+1}).
\end{align*}
As noted in \cite[Chp. 3, Thm. 3.4]{BeLiPa-78}, this is equivalent to determining the set of solutions to 
\begin{align*}
	Lz=0\ \text{with}\ z\in W^{2,2}_{per}([0,1]^{d+1}).
\end{align*}
Just as in \cite[Chp. 3, Thm. 3.4]{BeLiPa-78}, one can use a boot strapping argument to raise the exponent in the $L^p$ estimate for $z$ to show that in fact, any such $z$ is in $W^{2,(d+1)}_{per}([0,1]^{d+1})$.  At this point, we see that $z$ extends to a global, periodic, bounded, and $W^{2,(d+1)}_{loc}$ solution of $Lz=0$.  Applying the Krylov-Safonov Theorem in, e.g. $B_1$, to the function, $z_R(x)=z(Rx)$, we see
\begin{align*}
	[z_R]_{C^\gam(B_1)}\leq C\norm{z_R}_{L^\infty(B_2)},
\end{align*}
and hence may conclude that $z$ must be a constant.  Thus, there is a one parameter family of solutions to the adjoint equation that differ by a multiplicative constant, and $m$ is selected as the unique one that is positive and has $\int_{[0,1]^{d+1}}m=1$.  We note that the positivity of $m$ follows as in \cite[Chp. 3, Thm. 3.4]{BeLiPa-78}.

\subsection{The perturbed test function method for the interior homogenization of (\ref{eqIntro:MainEpScale})}\label{sec:Perturbed}

Just for clarity, we include the main arguments of the perturbed test function method for $u^\ep$.  We did not intend to give a complete presentation of the material, but rather just confirm that the results remain intact when using the perturbed test function method to show the homogenization of (\ref{eqIntro:MainEpScale}) instead of the method of classical solutions given in \cite{BeLiPa-78}.  What follows is a summary of the methods of \cite[Chp. 3, sec 4.2, 5.1]{BeLiPa-78}, combined with the perturbed test function method in \cite{Evan-92PerHomog}, and as such we suggest the interested reader consult both \cite{BeLiPa-78} and our comments to get the complete description.  

We begin with the formal ansatz that assumes
\begin{align*}
	u^\ep(x) = \bar u(x,\frac{x}{\ep}) + \ep v(x,\frac{x}{\ep}) + \ep^2 w(x,\frac{x}{\ep}) + o(\ep^2).
\end{align*} 
Plugging this into (\ref{eqIntro:MainEpScale}) and ignoring terms of the size $O(\ep)$, we see that

\begin{align}
	&\Tr(A(\frac{x}{\ep})D^2u^\ep) + \frac{1}{\ep}B(\frac{x}{\ep})\cdot \grad u^\ep
	\nonumber\\
	&=A_{ij}(\frac{x}{\ep})\bar u_{x_ix_j}(x,\frac{x}{\ep}) + \frac{1}{\ep}A_{ij}(\frac{x}{\ep})\left(\bar u_{x_iy_j}(x,\frac{x}{\ep}) + \bar u_{y_ix_j}(x,\frac{x}{\ep})\right) 
	+ \frac{1}{\ep^2}A_{ij}(\frac{x}{\ep})\bar u_{y_iy_j}(x,\frac{x}{\ep})
	\nonumber\\
	&+ \ep A_{ij}(\frac{x}{\ep})v_{x_ix_j}(x,\frac{x}{\ep}) + A_{ij}(\frac{x}{\ep})\left(v_{x_iy_j}(x,\frac{x}{\ep}) + v_{y_ix_j}(x,\frac{x}{\ep})\right) + \frac{1}{\ep}A_{ij}(\frac{x}{\ep})v_{y_iy_j}(x,\frac{x}{\ep})
	\nonumber\\
	&+ \ep^2 A_{ij}(\frac{x}{\ep})w_{x_ix_j}(x,\frac{x}{\ep}) 
	+ \ep A_{ij}(\frac{x}{\ep})\left(w_{x_iy_j}w(x,\frac{x}{\ep}) 
	+w_{y_ix_j}(x,\frac{x}{\ep}) \right)
	+ A_{ij}(\frac{x}{\ep})w_{y_iy_j}(x,\frac{x}{\ep})
	\nonumber\\
	&+ \frac{1}{\ep}B(\frac{x}{\ep})\cdot\grad_x \bar u(x,\frac{x}{\ep}) + \frac{1}{\ep^2}B(\frac{x}{\ep})\cdot\grad_y \bar u(x,\frac{x}{\ep})
	\nonumber\\
	&+ B(\frac{x}{\ep})\cdot\grad_x v(x,\frac{x}{\ep}) + \frac{1}{\ep}B(\frac{x}{\ep})\cdot\grad_y v(x,\frac{x}{\ep})
	\nonumber\\
	&+ \ep B(\frac{x}{\ep})\cdot\grad_x w(x,\frac{x}{\ep}) + B(\frac{x}{\ep})\cdot\grad_y w(x,\frac{x}{\ep})
	 \label{eqPerturbed:MegaClusterEpsilon}
\end{align}

In order for there to be any hope of extracting a limit from this equation, we try to collect the negative powers of $\ep$ and set them to zero.  First, we have
\begin{align*}
	\frac{1}{\ep^2}A_{ij}(\frac{x}{\ep})\bar u_{y_iy_j}(x,\frac{x}{\ep})
	+ \frac{1}{\ep^2}B(\frac{x}{\ep})\cdot\grad_y \bar u(x,\frac{x}{\ep})
	=0
\end{align*}
This implies that we can choose, with an abuse of notation,
\begin{align*}
	\bar u(x,\frac{x}{\ep})=\bar u(x).
\end{align*}
Hence both $\bar u_{x_iy_j}=0$ and $\bar u_{y_ix_j}=0$.

The $1/\ep$ terms now become
\begin{align}\label{eqPerturb:BadV}
	\frac{1}{\ep}A_{ij}(\frac{x}{\ep})v_{y_iy_j}(x,\frac{x}{\ep})
	+ \frac{1}{\ep}B(\frac{x}{\ep})\cdot\grad_x \bar u(x,\frac{x}{\ep})
	+ \frac{1}{\ep}B(\frac{x}{\ep})\cdot\grad_y v(x,\frac{x}{\ep})
	=0
\end{align}
We let $p(x)=\grad_x\bar u(x)$.  We need to find a $v$ so that
\begin{align}\label{eqPerturb:EqForVandP}
	A_{ij}(y)v_{y_iy_j}(x,y)
	+ B(y)\cdot\grad_y v(x,y) = -B(y)\cdot p(x).
\end{align} 
Let $\chi(y)=(\chi^1(y),\dots,\chi^{d+1}(y))^T$ solve
\begin{align}\label{eqPerturb:ChiCorrector}
	A_{ij}(y)\chi^l_{y_jy_j} + B(y)\grad_y\chi^l(y) = -B^l(y),
\end{align}
which is possible because of Assumption \ref{assume:ABCompatible},
\begin{align*}
	\int_{[0,1]^{d+1}}B^l(y)m(y)dy=0.
\end{align*}
We note that $\chi^l$ are independent of $\bar u$!
Suppose that 
\begin{align}\label{eqPerturb:VCorrector}
	v(x,y) = \sum_l \bar u_{x_l}(x)\chi^{l}(y) + \tilde v(x).
\end{align}
We note that in the previous and upcoming equation, $x$, is just a parameter which can be considered fixed, hence the notation $p=\grad \bar u(x)$.  However, the $x$ dependence is relevant when searching for the function, $w(x,y)$.
Then we have
\begin{align*}
	&A_{ij}(y)v_{y_iy_j}(y) + \sum_i B^i(y)v_{y_i}(y)\\
	&= \sum_l p_l(x) A_{ij}(y)\chi^l_{y_iy_j}(y) + \sum_{l,i} p_l(x)B^i(y)\chi^l_{y_i}(y)\\
	&= \sum_l p_l(x)\left(A_{ij}(y) \chi^l_{y_iy_j}(y) + \sum_i B^i(y)\chi^l_{y_i}(y) \right)\\
	&= \sum_l p_l(x)\left(A_{ij}(y) \chi^l_{y_iy_j}(y) + B(y)\cdot\grad_y\chi^l(y) \right)\\
	&= \sum_l p_l(x)\left(-B^l(y) \right)\\
	&= -B(y)\cdot p(x)
\end{align*}

Thus far we have now identified the function $v$, depending on $\bar u$.  In what follows, we need to show that there is a particular choice of effective coefficients, $\bar a_{ij}$ which will yield the existence of $w$.  It is this compatibility condition for $w$, which involves the first corrector, $v$, that gives the effective equation for $\bar u$.  After ignoring all terms with $\ep$ or $\ep^2$ (which will be transparent from the proof of the perturbed test function method in the proof of Proposition \ref{prop:PTFM} below), what is left of 
\begin{align*}
	\Tr(A(\frac{x}{\ep})D^2u^\ep) + \frac{1}{\ep}B(\frac{x}{\ep})\cdot \grad u^\ep=0,
\end{align*}
now reads as
\begin{align*}
	&A_{ij}(\frac{x}{\ep})\bar u_{x_ix_j}(x)
	+A_{ij}(\frac{x}{\ep})\left(v_{x_iy_j}(x,\frac{x}{\ep}) + v_{y_ix_j}(x,\frac{x}{\ep})\right)
	+ A_{ij}(\frac{x}{\ep})w_{y_iy_j}(x,\frac{x}{\ep})\\
	&\ \ \ \ + B(\frac{x}{\ep})\cdot\grad_x v(x,\frac{x}{\ep})
	+  B(\frac{x}{\ep})\cdot\grad_y w(x,\frac{x}{\ep})
	=0.
\end{align*}
Considering $x$ fixed, setting $Q=D^2\bar u(x)$, and inserting $v$, we seek a $w$ that satisfies
\begin{align*}
	& A_{ij}(y)w_{y_iy_j}(x,y) +  B(y)\cdot\grad_y w(x,y)\\
	&= -A_{ij}(y)Q_{ij} 
	-A_{ij}(\frac{x}{\ep})\left(v_{x_iy_j}(x,\frac{x}{\ep}) + v_{y_ix_j}(x,\frac{x}{\ep})\right)
	- B(y)\cdot\grad_x v(x,y),
\end{align*}
or in terms of $\chi$,
\begin{align*}
	& A_{ij}(y)w_{y_iy_j}(x,y) +  B(y)\cdot\grad_y w(x,y)\\
	&= -A_{ij}(y)Q_{ij} 
	-A_{ij}\left(\bar u_{x_lx_i}(x)\chi^l_{y_j}(y) +\bar u_{x_lx_j}(x)\chi^l_{y_i}(y)\right)  \\
	&\ \ \ \ \ - B^k(y)(\bar u_{x_l}\chi^l(y)+\tilde v(x))_{x_k}.
\end{align*}
This further simplifies as
\begin{align*}
	& A_{ij}(y)w_{y_iy_j}(x,y) +  B(y)\cdot\grad_y w(x,y)\\
	&= -A_{ij}(y)Q_{ij} 
	-A_{ij}(y)\bar u_{x_lx_i}(x)\chi^l_{y_j}(y) 
	-A_{ij}(y)\bar u_{x_lx_j}(x)\chi^l_{y_i}(y) \\
	&\ \ \ \ \  - B^k(y)\bar u_{x_lx_k}(x)\chi^l(y) -B^k(y)\tilde v_{x_k}(x).
\end{align*}
Since we only have an equation in the $y$ variable for $w$, in the best case scenario, we would seek a periodic $w$, solving
\begin{align*}
	& A_{ij}(y)w_{y_iy_j}(x,y) +  B(y)\cdot\grad_y w(x,y)\\
	&= -A_{ij}(y)Q_{ij} 
	-A_{ij}Q_{li}\chi^l_{y_j}(y) 
	 -A_{ij}(y)Q_{lj}\chi^l_{y_i}(y)\\
	&\ \ \ \ \ - B^k(y)Q_{lk}\chi^l(y) -B^k(y)\tilde v_{x_k}(x).
\end{align*}
This is, of course, too strict.  Thus, to relax the problem, we seek a unique choice of $\lam(Q)$ to balance the right hand side.  That is, we seek a unique choice of $\lam(Q)$ such that there exists a $w$ solving
\begin{align}
	& A_{ij}(y)w_{y_iy_j}(x,y) +  B(y)\cdot\grad_y w(x,y)\nonumber\\
	&= -A_{ij}(y)Q_{ij} 
	-A_{ij}Q_{li}\chi^l_{y_j}(y) 
	 -A_{ij}(y)Q_{lj}\chi^l_{y_i}(y)\nonumber\\
	&\ \ \ \ \ - B^k(y)Q_{lk}\chi^l(y) -B^k(y)\tilde v_{x_k}(x) + \lam(Q).\label{eqPerturb:EqForWCorrector}
\end{align}
Since we already know that 
\begin{align*}
	\int_{[0,1]^{d+1}}B^k(y)m(y)dy=0,
\end{align*}
such a $w$ can exist only if 
\begin{align*}
	&\int_{[0,1]^{d+1}}\lam(Q)m(y)dy\\
	&-\int_{[0,1]^{d+1}}\left(
	A_{ij}(y)Q_{ij} 
		+A_{ij}(y)Q_{lj}\chi^l_{y_j}(y)  
		+A_{ij}(y)Q_{li}\chi^l_{y_i}(y)
		+ B^k(y)Q_{lk}\chi^l(y)
	\right)m(y)dy
	=0
\end{align*}
We note that this can be re-written in new index variables as
\begin{align*}
	&\int_{[0,1]^{d+1}}\lam(Q)m(y)dy\\
	&-\int_{[0,1]^{d+1}}
	\Big(A_{mn}(y)Q_{mn} 
		+A_{pm}(y)Q_{mn}\chi^n_{y_p}(y)  
		+A_{pn}(y)Q_{mn}\chi^m_{y_p}(y)\\
		&\ \ \ \ \ \ \ \ \ \ + \frac{1}{2}(B^m(y)Q_{mn}\chi^n(y)+B^n(y)Q_{mn}\chi^m(y))\Big)
	m(y)dy
	=0
\end{align*}
We know that $\lam(Q)$ will be linear in $Q$, so we will use the form
\begin{align*}
	\lam(Q)=\bar a_{mn}Q_{mn},
\end{align*}
and hence we need  
\begin{align*}
	&Q_{mn}\int_{[0,1]^{d+1}}\bar a_{mn}m(y)dy\\
	&-Q_{mn}\int_{[0,1]^{d+1}}
	\Big(A_{mn}(y) 
		+A_{pm}(y)\chi^n_{y_p}(y)  
		+A_{pn}(y)\chi^m_{y_p}(y)\\
		&\ \ \ \ \ \ \ \ \ \ \  + \frac{1}{2}(B^m(y)\chi^n(y)+B^n(y)\chi^m(y))\Big)
	m(y)dy
	=0.
\end{align*}
We see that $\bar a_{mn}$ must be uniquely chosen as
\begin{align}
	\bar a_{mn} = &\int_{[0,1]^{d+1}}
	\Big(A_{mn}(y) 
		+A_{pm}(y)\chi^n_{y_p}(y)  
		+A_{pn}(y)\chi^m_{y_p}(y)\nonumber\\
		&\ \ \ \ + \frac{1}{2}(B^m(y)\chi^n(y)+B^n(y)\chi^m(y))\Big)
	m(y)dy.\label{eqPerturb:ABarDef}
\end{align}

All of this work can be summarized in the following proposition.

\begin{prop}\label{prop:CorrectorAppendix}
	Assume that $\phi\in C^{2}$, $p(x)=\grad\phi(x)$, and that $Q\in\mathcal S(d+1)$ is fixed. Define $\chi^l$ and $v$ respectively by using $\phi_{x_l}$ in (\ref{eqPerturb:ChiCorrector}) and (\ref{eqPerturb:VCorrector}).  There exists a unique choice of $\lam(Q)\in\real$ such that (\ref{eqPerturb:EqForWCorrector}) admits a (classical) periodic solution, $w$, as a function of $y$.  In this instance, $\lam(Q)$ is computed explicitly in terms of $\chi$ via (\ref{eqPerturb:ABarDef}), with
\begin{align*}
	\lam(Q)=\bar a_{mn}Q_{mn}.
\end{align*}
\end{prop}

\begin{rem}
	The previous statement is just a summary of the steps that culminate on \cite[p. 416]{BeLiPa-78}.  
\end{rem}

\begin{rem}
	We note that (\ref{eqPerturb:EqForWCorrector}) indicates that $\tilde v$ can be any reasonable function of $x$, and $w$ can still be determined.  Thus we are free to set, a posteriori, $\tilde v(x)\equiv 0$, hence making $w$ a function of only $y$.
\end{rem}

Now we move onto implementing the perturbed test function method for this equation.  It is just a rewriting of the details in \cite{Evan-92PerHomog} in our context.  We claim that

\begin{prop}\label{prop:PTFM}
	Assume that $\bar u$ is any local uniform limit of $u^\ep$ in $\Sigma^1$.  Then $\bar u$ must be a solution of 
	\begin{align}\label{eqPerturb:Effective}
		\bar a_{ij}\bar u_{x_ix_j}=0\ \text{in}\ \Sigma^1,
	\end{align}
	with $\bar a_{ij}$ defined in (\ref{eqPerturb:ABarDef}).
\end{prop}

\begin{proof}
We will prove that $\bar u$ is a (viscosity) subsolution of (\ref{eqPerturb:Effective}).  Similarly one establishes that $\bar u$ is a (viscosity) supersolution.  We will proceed by contradiction and assume that $\bar u$ is not a viscosity subsolution.  That is, we assume that $\phi$ is smooth and bounded and that $\bar u-\phi$ attains a strict local max in $B_r(x_0)$ at $x_0$, but for some $\del>0$
\begin{align*}
	\bar a_{ij}\phi_{x_ix_j}(x_0)\leq-\del<0.
\end{align*}  
Let us take
\begin{align*}
	p(x)=\grad\phi(x)\ \text{and}\ Q=D^2\phi(x_0),
\end{align*}
and let $\chi^l$, $v$, and $w$ be as in Proposition \ref{prop:CorrectorAppendix}.  We claim that 
\begin{align*}
	\psi^\ep(x) = \phi(x) + \ep v(x,\frac{x}{\ep}) + \ep^2w^\ep(\frac{x}{\ep})
\end{align*}
is in fact a viscosity supersolution of
\begin{align*}
	\Tr(A(\frac{x}{\ep})D^2\psi^\ep) + \frac{1}{\ep}B(\frac{x}{\ep})\cdot\grad\psi^\ep=0\ \text{in}\ B_{\rho}(x_0),
\end{align*}
for some $\rho>0$ appropriately small, depending on $\norm{\phi}_{C^{2,\gam}(B_r(x_0))}$, $\norm{v}_{C^{1,1}}$, and $\norm{w}_{C^{1,1}}$ (in fact, $\psi^\ep$ is a classical solution, but we only care about a class of solutions that satisfies a comparison theorem).  We note that the contradiction assumption for $\phi$ can be restated as
\begin{align*}
	\lam(D^2\phi(x_0))\leq -\del.
\end{align*}
We also note by the boundedness of $A$, $B$, $\chi^l$, and $\grad\chi^l$, that for all $x\in B_\rho(x_0)$, we can effectively localize the equation at $x_0$ because
\begin{align*}
	|&A_{ij}(\frac{x}{\ep})\phi_{x_ix_j}(x) 
		+A_{ij}(\frac{x}{\ep}) \phi_{x_lx_i}(x)\chi^l_{y_j}(\frac{x}{\ep}) 
		+A_{ij}(\frac{x}{\ep}) \phi_{x_lx_j}(x)\chi^l_{y_i}(\frac{x}{\ep})
		+B^k(\frac{x}{\ep})\phi_{x_lx_k}(x)\chi^l(\frac{x}{\ep})\\
	&	-\left(A_{ij}(\frac{x}{\ep})\phi_{x_ix_j}(x_0) 
		+A_{ij}(\frac{x}{\ep}) \phi_{x_lx_i}(x_0)\chi^l_{y_j}(\frac{x}{\ep}) 
		+A_{ij}(\frac{x}{\ep}) \phi_{x_lx_j}(x_0)\chi^l_{y_i}(\frac{x}{\ep})
		+B^k(\frac{x}{\ep})\phi_{x_lx_k}(x_0)\chi^l(\frac{x}{\ep})
		\right)|\\
		&\leq \frac{\del}{4}.
\end{align*}
Furthermore, we can possibly restrict $\rho$ to be smaller, depending upon the $C^{1,1}$ norms of $v$ and $w$ so that 
\begin{align*}
	&|
	\ep A_{ij}(\frac{x}{\ep})v_{x_ix_j}(\frac{x}{\ep})
	+ \ep^2 A_{ij}(\frac{x}{\ep})w_{x_ix_j}(\frac{x}{\ep}) 
	+\ep A_{ij}(\frac{x}{\ep})\left(w_{x_iy_j}w(\frac{x}{\ep}) 
		+w_{y_ix_j}(\frac{x}{\ep}) \right) \\
	&\ \ \ \ \ \ \ \ \ \ + \ep B(\frac{x}{\ep})\cdot\grad_x w(\frac{x}{\ep})
	|
	\leq \frac{\del}{4}
\end{align*}
Hence, plugging $\psi^\ep$ into (\ref{eqPerturbed:MegaClusterEpsilon}), using the particular choices of $v$ and $w$, and inspecting, we see that in $B_\rho(x_0)$
\begin{align*}
	&\Tr(A(\frac{x}{\ep})D^2\psi^\ep) + \frac{1}{\ep}B(\frac{x}{\ep})\cdot\grad\psi^\ep\\
	&\leq A_{ij}(y)w_{y_iy_j}(x,y) +  B(y)\cdot\grad_y w(x,y)
	+A_{ij}(y)Q_{ij} 
	+A_{ij}Q_{li}\chi^l_{y_j}(y) 
	 +A_{ij}(y)Q_{lj}\chi^l_{y_i}(y)\\
	 &\ \ \ \ + B^k(y)Q_{lk}\chi^l(y) 
	 +B^k(y)\tilde v_{x_k}(x)
	 +\frac{\del}{4}+ \frac{\del}{4}\\
	 &\leq \lam(Q) +\frac{\del}{2}
	 <0.
\end{align*}
We importantly note that by the construction of $v$, we had exact equality in the equation (\ref{eqPerturb:EqForVandP}), thus canceling these $1/\ep$ terms above, which are the same as in (\ref{eqPerturb:BadV}).
Thus we conclude by the comparison of sub and super solutions that
\begin{align*}
	u^\ep(x_0)-\phi^\ep(x_0)\leq \max_{\overline{B_\rho(x_0)}}u^\ep - \psi^\ep \leq \max_{\partial B_\rho(x_0)} u^\ep-\psi^\ep.
\end{align*}
Now we note that $u^\ep\to \bar u$ and $\psi^\ep\to\phi$ uniformly in $\overline{B_\rho}(x_0)$.  Thus
\begin{align*}
	\bar u(x_0)-\phi(x_0)\leq  \max_{\partial B_\rho(x_0)} \bar u-\phi,
\end{align*}
which is a contradiction to the strict max at $x_0$.  Thus, we see that in fact $\bar u$ is a subsolution of (\ref{eqPerturb:Effective}).  This concludes the proof of proposition \ref{prop:PTFM}.
\end{proof}

\subsection{The rate of convergence for the regular homogenization}\label{sec:Rates}

Here we mention how some minor modifications to the arguments in \cite[Chp. 3, Sec 5]{BeLiPa-78} yield the rate of convergence under our assumptions on $A$ and $B$.  

First, we note that in the proof of Proposition \ref{prop:PTFM}, the expansion 
\begin{align*}
	\psi^\ep(x) = \phi(x) + \ep v(x,\frac{x}{\ep}) + \ep^2w^\ep(\frac{x}{\ep})
\end{align*}
is only used locally, and we did not utilize $x$ dependence for $w$ or the values of $\bar u$ (but $\bar u$ implicitly plays a role through the $\phi$).  However, as pointed out on \cite[p.418-419]{BeLiPa-78}, you can get much more information out of this expansion by using a better choice for $w$.  

In all that follows, \cite[p.418-419]{BeLiPa-78} have more coefficients in $L$ than we do (and they call their operator $A$).  In the context of their notation, we have for $A$ and $B$ given in this work, and $a$, $b$, $c$, $a_0$ in \cite{BeLiPa-78},
\begin{align*}
	a_{ij}(y)=A_{ij}(y),\ b_i(y)=B_i(y),\ c_i(y)=a_0\equiv 0.
\end{align*}

Assume that $w^\ep$ and $\bar w$ are as in the statement of Proposition \ref{prop:RatesHomogGlobal}; here $w^\ep$ and $\bar w$ play the role of respectively $u^\ep$ and $u$ in \cite[p.418-419]{BeLiPa-78}.  We note that $\bar A$ is a constant coefficient and uniformly elliptic matrix, and thus $\bar w$ is locally smooth and globally as smooth as is $f$ (the Dirichlet data), in particular $\norm{\bar w}_{C^{4,\gam}(\overline{\Sigma^1})}\leq C\norm{f}_{C^{4,\gam}(\Sigma_0)}$.
The good expansion is
\begin{align*}
	\tilde w^\ep(x) = \bar w(x) + \ep v(x,\frac{x}{\ep}) + \ep^2 w_2(x,\frac{x}{\ep}),
\end{align*}
where $v$ is defined in (\ref{eqPerturb:EqForVandP}) using $p(x)=\grad \bar w(x)$, and $w_2$ is defined as
\begin{align*}
	w_2(x,y) = \bar w_{x_ix_j}(x)\chi^{ij}(y),
\end{align*}
and $\chi^{ij}$ (note these are different than $\chi^i$ with one upper index) are chosen to solve
\begin{align*}
	&\Tr(A(y)D^2\chi^{ij}(y)) + B(y)\cdot\grad\chi^{ij}(y) = \\
	&\ \ \ \ \ \ \ \ \ \ \bar a_{ij} -\{A_{ij}(y) 
		+A_{kj}(y)\chi^i_{y_k}(y)  
		+A_{ki}(y)\chi^j_{y_k}(y)
		+ \frac{1}{2}(B^i(y)\chi^j(y)+B^j(y)\chi^i(y))\}.
\end{align*}
Now, we note that $A$ and $B$ periodic and in $C^\gam$ implies that $\chi^l$, and as a result also $\chi^{ij}$, are all periodic and $C^{2,\gam}(\Sigma^1)$.

The final step is to compute the equation for 
\begin{align*}
	z^\ep = w^\ep - \tilde w^\ep,
\end{align*}
which can be followed directly in \cite[p. 418, eq. (5.23)]{BeLiPa-78}, for
\begin{align*}
	L(w^\ep - \tilde w^\ep) = \ep g^\ep.
\end{align*}
We use the same definition of $g^\ep$ as in \cite[p.418]{BeLiPa-78} (note here it is significantly simpler due to the absence of $c_i$ and $a_0$).  The important thing to observe about $g^\ep$ is that it involves: the coefficients, $A$ and $B$; the function $\chi^l$; the function and up to one derivative of $\chi^{ij}$; third and fourth derivatives of $\bar w$.  Thus by the regularity that is noted above, $\norm{g^\ep}_{L^\infty}\leq C$, independently of $\ep$.  Furthermore, $(w^\ep-\tilde w^\ep)|_{\Sigma_0\union \Sigma_1}\leq C\ep$ (which can be checked by a simple calculation).  Thus we conclude
\begin{align*}
	\norm{z^\ep}_{L^\infty(\Sigma^1)}\leq \ep C,
\end{align*} 
and this implies
\begin{align*}
	\norm{w^\ep-\bar w}_{L^\infty(\Sigma^1)}\leq  \ep C.
\end{align*}

\subsection{The Divergence Structure of Equation (\ref{eqIntro:MainEpScale})}\label{sec:DivergenceFormEq}

As mentioned in the introduction, one of the anonymous referees pointed out that the standard assumption surrounding the interior homogenization of (\ref{eqIntro:MainEpScale}), that is Assumption \ref{assume:ABCompatible}, means that equation (\ref{eqIntro:MainEpScale}) can be re-written in a divergence form, with a possibly non-symmetric coefficient matrix.  This trick of rewriting non-divergence equations in divergence form in homogenization also appears in \cite[Sec. 1]{AvellanedaLin-1989CompactnessHomogNonDivEqCPAM} and \cite[Chp. 3, Sec. 5.2]{BeLiPa-78}.  Since our approach was inherently non-variational, based on comparison principle arguments, we did not pursue an alternate proof to utilize this hidden divergence structure of (\ref{eqIntro:MainEpScale}).  However, in moving to other equations or possibly looking into the case of random environments, it is conceivable that this divergence structure could lend some insights or advantages, such as in regards to rates of convergence like \cite{ArmstrongSmart-2016QuantHomogConvexEnergyFunASENS}. 

Here we briefly mention the details that make this conversion possible.  It is not new, and we just summarize the details presented on the same topic in \cite[Sec. 1]{AvellanedaLin-1989CompactnessHomogNonDivEqCPAM}, where they do the same transformation for an equation with no drift.  We start with the operator, $L$, given by
\[
L(u)=A_{ij}(y)u_{y_iy_j}+B(y)\cdot \grad u=A_{ij}(y)u_{y_iy_j}+B_j(y)u_{y_j},
\]
and, as above, its formal adjoint,
\[
L^*(v)= \left((A_{ij}(y)v)_{y_i}-vB_j(y) \right)_{y_j}.
\]
As shown in Appendix \ref{sec:InvariantMeasure}, we know that there is a unique invariant measure, $m$, which is a \emph{positive} solution to the adjoint equation.  The goal \emph{would be} to find some choice of matrix, $C_{ij}$, so that we can represent
\[
L(u) = \Div(C_{ij}\grad u),
\]
but this will not work except in special cases.  However, using the invariant measure, $m$, we will be able to find $C_{ij}$ so that
\[
mL(u)=\Div(C_{ij}\grad u).
\]
Using the fact that $m$ solves $L^*(m)=0$, we have that
\[
\sum_i\sum_j\left((A_{ij}(y)m)_{y_i}-mB_j(y) \right)_{y_j}=0,
\]
and so we can define the vector field $F=(F_j)_{j=1}^{d+1}$, via the relationship
\[
-\sum_i(A_{ij}(y)m)_{y_i}+mB_j(y)= F_j.
\]
Expanding the relationship between $mL(u)$ and $\Div(C_{ij}\grad u)$, it is clear the we seek $C_{ij}$ of the form
\[
C_{ij}(y)=m(y)A_{ij}(y) + T_{ij},
\] 
where we have yet to determine $T$.  We will define $T_{ij}$ via potentials, $\phi_j$. We take $\phi_j$ defined as
\[
\Delta \phi_j=F_j,\ \ \phi_j\ \text{periodic},\ \ \int_{[0,1]^{d+1}}\phi_jdy=0.
\]
This is possible because as derivatives of periodic functions, $(m(y)A_{ij})_{y_i}$, have mean zero, and also $m(y)B_j(y)$ has mean zero by Assumption \ref{assume:ABCompatible}. Now, we define
\[
T_{ij}(y)=(\phi_j)_{y_i}-(\phi_i)_{y_j}.
\]

The most useful properties of $T_{ij}$ are that
\[
T_{ij}+T_{ji}=0,\ \ \sum_i(\phi_i)_{y_i}=0,\ \ \text{and}\ \ \sum_i (T_{ij})_{y_i}=\Delta \phi_j.
\]
This can be used to show that
\begin{align*}
\sum_i\sum_j(C_{ij}u_{y_j})_{y_i}&=\sum_i\sum_j mA_{ij}u_{y_iy_j}+(mA_{ij})_{y_i}u_{y_j}\\
&\ \ \ \ \ \ \ \ 
+\sum_i\sum_j T_{ij}u_{y_iy_j} + (T_{ij})_{y_i}u_{y_j}.
\end{align*}
We see that the third term on the right vanishes by construction of the anti-symmetry of $T_{ij}$, and the second and fourth terms are constructed so that
\[
\sum_i(mA_{ij})_{y_i}+(T_{ij})_{y_i}=\sum_i(mA_{ij})_{y_i}+\Delta\phi_j
=\sum_i(mA_{ij})_{y_i}-(A_{ij}(y)m)_{y_i}+mB_j(y)=mB_j
\]
Thus, this construction has showed
\[
\sum_i\sum_j(C_{ij}u_{y_j})_{y_i}=\sum_i\sum_j mA_{ij}u_{y_iy_j}+mB_j(y)u_{y_j},
\]
which is exactly the divergence structure that was desired.


\bibliography{../refs}
\bibliographystyle{plain}
\end{document}